\documentclass[11pt, reqno]{amsart}

\usepackage[dvipsnames]{xcolor}

\usepackage{amssymb}
\usepackage{verbatim}
\usepackage{float}
\usepackage{hyperref}
\usepackage{tikz}
\usepackage{caption}
\usepackage{hyperref}
\usepackage[left=1.1in, right=1.1in, top=1in, bottom=.9in]{geometry}
\usepackage{times}
\usepackage[T1]{fontenc}
\usepackage{mathrsfs}
\usepackage{epsfig}
\usepackage{color}
\usepackage{array} 
\renewcommand{\arraystretch}{1.25}
\usepackage{enumitem}

\usepackage{multirow}

\usepackage{longtable}

\usepackage{diagbox}

\usetikzlibrary{arrows.meta,calc,decorations.markings,math}

\newtheorem{thm}{Theorem}[section]

\newtheorem{lemma}[thm]{Lemma}

\newtheorem{cor}[thm]{Corollary}
\newtheorem{prop}[thm]{Proposition}

\theoremstyle{remark}

\newtheorem*{remark}{Remark}

\theoremstyle{definition}
\newtheorem{defn}[thm]{Definition}
\newtheorem{question}[thm]{Question}
\newtheorem{example}[thm]{Example}

\newcommand{\thistheoremname}{}
\newtheorem{genericthm}[thm]{\thistheoremname}
\newenvironment{namedthm}[1]
  {\renewcommand{\thistheoremname}{#1}%
   \begin{genericthm}}
  {\end{genericthm}}

\numberwithin{equation}{section}

\def\ZZ{\mathbb{Z}}
\def\QQ{\mathbb{Q}}
\def\PP{\mathbb{P}}

\def\AA{\mathbb{A}}

\def\FF{\mathbb{F}}

\def\RR{\mathbb{R}}
\def\CC{\mathbb{C}}

\def\C{\mathcal{C}}
\def\L{\mathcal{L}}

\renewcommand{\P}{\mathcal{P}}
\def\Q{\mathcal{Q}}
\def\F{\mathcal{F}}

\def\Poly{\mathrm{Poly}}
\def\im{\mathrm{im}}

\def\conf{\mathrm{Conf}}
\def\confhat{\widehat{\mathrm{Conf}}}
\def\endo{\mathrm{End}}
\def\F{\mathcal{F}}
\def\aut{\mathrm{Aut}}
\def\aff{\mathrm{Aff}}

\def\preper{\mathrm{PrePer}}

\def\wh{\widehat}

\def\PGL{\mathrm{PGL}}
\def\M{\mathcal{M}}
\def\Mhat{\widehat{\M}}

\begin{document}

\title[Dynamical moduli spaces]{Dynamical moduli spaces and\\ polynomial endomorphisms of configurations}

\author{Talia Blum}
\address{Department of Mathematics\\
Stanford University\\
Stanford, CA 94305}
\email{taliab@stanford.edu}

\author{John R. Doyle}
\address{Department of Mathematics\\
Oklahoma State University\\
Stillwater, OK 74078}
\email{john.r.doyle@okstate.edu}

\author{Trevor Hyde}
\address{Dept. of Mathematics\\
University of Chicago \\
Chicago, IL 60637
}
\email{tghyde@uchicago.edu}

\author{Colby Kelln}
\address{Department of Mathematics\\
Cornell University\\
Ithaca, NY, 14853}
\email{ck765@cornell.edu}

\author{Henry Talbott}
\address{Department of Mathematics\\
University of Michigan\\
Ann Arbor, MI, 48109}
\email{htalbott@umich.edu}

\author{Max Weinreich}
\address{Department of Mathematics\\
Brown University\\
Providence, RI 02906}
\email{max\_weinreich@brown.edu}

\maketitle

\section{Introduction}
\label{sec intro}

The study of preperiodic points of polynomials is of central importance in complex and arithmetic dynamics. Typically, one starts with a field $K$ and a polynomial $f(x) \in K[x]$ and then tries to understand the preperiodic points of $f(x)$; that is, the points $\alpha$ for which the {\it orbit}
$
    \{\alpha, f(\alpha), f(f(\alpha)), f(f(f(\alpha))), \ldots\}
$
is finite. We take a different perspective: given a configuration $q$ of finitely many points on the affine line, we seek to understand the collection of all polynomials $f(x)$ such that $f(q) \subseteq q$. To that end let
\[
    \conf^n := \{(q_1, q_2, \ldots, q_n) \in \AA^n : q_i \neq q_j \text{ for all } i \ne j\}
\]
denote the configuration space of $n$ distinct points on the affine line, and define 
\[
    \endo(q) := \{f(x) : f(q) \subseteq q\}
\]
to be the semigroup (with respect to composition) of all polynomials which stabilize $q \in \conf^n$ as a set.

\begin{remark}
For much of this article, we will allow $K$ to be an arbitrary field. In the instances where our statements only apply to certain fields, we will specify the appropriate base fields.
\end{remark}

Our general goal is to begin addressing the following question.

\begin{question}
\label{quest:intro semigroup}
How does the semigroup $\endo(q)$ vary with $q \in \conf^n$?
\end{question}

Lagrange interpolation implies the existence of a unique polynomial of degree at most $n-1$ interpolating any set-theoretic endomorphism of $q \in \conf^n$.
For this reason, we focus on the elements of $\endo(q)$ with degree less than $n$. Let $\endo_d(q)$ denote the degree-$d$ graded component of $\endo(q)$.

\begin{question}
\label{quest:intro max cardinality}
For $0 \leq d \leq n - 1$, how does $\endo_d(q)$ vary with $q \in \conf^n$? What is the maximal cardinality $E_{n,d}$ of $\endo_d(q)$ as $q$ ranges over $\conf^n$ and which configurations achieve it?
\end{question}

A \emph{portrait} on $[n] := \{1,\ldots,n\}$ is a function $\P : [n] \rightarrow [n]$. The space $\conf_{\P,d}$ of \emph{degree-$d$ realizations} of a portrait $\P$ on $[n]$ is the subspace of $\conf^n$ defined by
\[
    \conf_{\P,d} := \{q \in \conf^n : \text{there exists a degree-$d$ polynomial $f$ such that } f(q_i) = q_{\P(i)} \text{ for all } i\}.
\]

The affine group $\aff_1$ of linear polynomials $\ell(x) = ax + b$ acts naturally on $\conf^n$ by
\[
    \ell(q) := (\ell(q_1), \ldots, \ell(q_n)).
\]
Observe that the action of $\aff_1$ stabilizes $\conf_{\P,d}$: Indeed, if there exists a degree-$d$ polynomial $f$ such that $f(q_i) = q_{\P(i)}$ for all $i \in [n]$, then $\widetilde{f} := \ell \circ f \circ \ell^{-1}(x)$ is also a degree-$d$ polynomial such that $\widetilde{f}(\ell(q_i)) = \ell(q_{\P(i)})$ for all $i \in [n]$, so $\ell(q) \in \conf_{\P,d}$. The \emph{degree-$d$ portrait moduli space} of a portrait $\P$ is defined to be the quotient
\[
    \M_{\P,d} := \conf_{\P,d}/\aff_1.
\]
Since $\aff_1$ acts sharply $2$-transitively on $\AA^1$, $\M_{\P,d}$ has the following simple model:
\[
    \M_{\P,d} = \{(q_1,\ldots,q_n) \in \conf_{\P,d} : q_1 = 0 \text{ and } q_2 = 1\}.
\]

Observe that a configuration $q$ has several degree-$d$ endomorphisms precisely when $q$ lies in the intersection of several portrait realization spaces. Thus, Question \ref{quest:intro max cardinality} leads us to consider the spaces
\[
    \conf_{\P,\Q,d} := \conf_{\P,d} \cap \conf_{\Q,d} \quad\text{and}\quad \M_{\P,\Q,d} := \conf_{\P,\Q,d}/\aff_1.
\]
We refer to $\conf_{\P,d}$ and $\conf_{\P,\Q,d}$ collectively as {\it portrait realization spaces}, and to $\M_{\P,d}$ and $\M_{\P,\Q,d}$ as {\it portrait moduli spaces}.

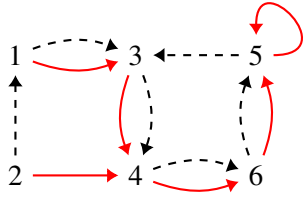
\begin{figure}[h]
    \centering
        
    \begin{tikzpicture}[shorten >= 0pt, scale=0.4]
    \node (00) at (0, 0) {2};
    \node (10) at (4, 0) {4};
    \node (11) at (4, 4) {3};
    \node (01) at (0, 4) {1};
    \node (20) at (8, 0) {6};
    \node (21) at (8, 4) {5};
    \draw[-{Latex[length=1.5mm,width=2mm]}, thick, red] (00) edge (10);
    \draw[-{Latex[length=1.5mm,width=2mm]}, thick, red] (01) edge [out=-20, in=-160] (11);
    \draw[-{Latex[length=1.5mm,width=2mm]}, thick, red] (11) edge [out=250, in=110] (10);
    \draw[-{Latex[length=1.5mm,width=2mm]}, thick, red] (10) edge [out=-20, in=-160] (20);
    \draw[-{Latex[length=1.5mm,width=2mm]}, thick, red] (20) edge [out=70, in=290] (21);
    \draw[-{Latex[length=1.5mm,width=2mm]}, thick, red] (21) edge [out=0, in=90, looseness = 7] (21);

    \path[-{Latex[length=1.5mm,width=2mm]}, thick, dashed] (00) edge (01);
    \path[-{Latex[length=1.5mm,width=2mm]}, thick, dashed] (01) edge [out=20, in=160] (11);
    \path[-{Latex[length=1.5mm,width=2mm]}, thick, dashed] (10) edge [out=20, in=160] (20);
    \path[-{Latex[length=1.5mm,width=2mm]}, thick, dashed] (11) edge [out=290, in=70] (10);
    \path[-{Latex[length=1.5mm,width=2mm]}, thick, dashed] (20) edge [out=110, in=250] (21);
    \path[-{Latex[length=1.5mm,width=2mm]}, thick, dashed] (21) edge (11);
    \end{tikzpicture}
    \caption{An illustration of portraits $\P$ (red, solid) and $\Q$ (black, dashed) acting on $\{1,2,3,4,5,6\}$. For example, the red, solid arrow from 2 to 4 indicates that $\P(2) = 4$. }
    \label{fig:portrait_example2}
\end{figure}

The isomorphism classes of the moduli spaces $\M_{\P,d}$ and $\M_{\P,\Q,d}$ depend only on the combinatorial type of the portrait $\P$ and the portrait pair $\{\P,\Q\}$, respectively; see Proposition \ref{prop:conjugation_isomorphism} for a more precise statement. We therefore ask the following:

\begin{question}
\label{quest: combinatorics}
How do combinatorial properties of portraits $\P$ and $\Q$ determine geometric properties of the moduli spaces $\M_{\P,d}$ and $\M_{\P,\Q,d}$?
\end{question}

In this paper, we initiate the study of these questions guided by computational results in low degrees.

\subsection{Results} We briefly summarize the main results of this article.

\subsubsection{The geometry of $\M_{\P,d}$ and $\M_{\P,\Q,d}$}
Our first result provides a combinatorial characterization of portraits $\P$ for which the 
moduli space $\M_{\P,d}$ 
is nonempty and achieves the expected dimension.

\begin{thm}
\label{thm intro single portrait}
Let $K$ be an algebraically closed field of characteristic $0$. Let $n\geq 2$ and $d\geq 2$ be integers, and let $\P : [n] \rightarrow [n]$ be a portrait. Then $\M_{\P,d}(K) \neq \emptyset$ if and only if
\begin{enumerate}
    \item Every element of $[n]$ has at most $d$ preimages under $\P$, and
    \item For every integer $k\geq 1$, $\P$ has at most $\frac{1}{k}\sum_{j\mid k}\mu(k/j)d^{j}$ periodic cycles of length $k$.
\end{enumerate}
In this case, $\dim \M_{\P,d} = \min\{d - 1, n - 2\}$.
\end{thm}

\begin{remark}
One direction of Theorem~\ref{thm intro single portrait} follows almost immediately from Proposition 15.1 and Theorem 15.8 of \cite{doyle/silverman}. More precisely, the cited results of \cite{doyle/silverman} are used in the proof of Theorem~\ref{thm intro single portrait} to show that if $d \ge n - 1$ and $\M_{\P,d}$ is nonempty, then conditions (1) and (2) are satisfied and $\dim \M_{\P,d} = d-1$; see the proof of Proposition~\ref{prop admissible}. The other direction of the statement does not follow directly from the results of \cite{doyle/silverman}, nor does the statement in the case that $d \le n - 2$.
\end{remark}

We will call a portrait $\P$ satisfying the two conditions in Theorem \ref{thm intro single portrait} an \emph{admissible degree-$d$ portrait.} If $\P$ and $\Q$ are admissible degree-$d$ portraits on $n$ points, then a simple count of parameters and constraints suggests that the dimension of $\M_{\P,\Q,d}$ should be $2d - n$ (see Section \ref{sec:data}). 
There are finitely many admissible degree-$d$ portraits---hence finitely many pairs $\{\P,\Q\}$ of such portraits---on a set with $n$ elements, so in principle one can survey all of the portrait moduli spaces $\M_{\P,\Q,d}$ for any fixed $n$ and $d$. We have conducted this survey for $(n,d) = (4,2)$ and $(n,d) = (6,3)$; in these instances, the expected dimension is zero.

For computational reasons, it is simpler to work with the moduli spaces 
\[
    \wh{\M}_{\P,\Q,d} := \bigcup_{e\leq d} \M_{\P,\Q,e} \subseteq \conf^n/\aff_1
\]
of all affine equivalence classes of degree-at-most-$d$ realizations of $\P$ and $\Q$.
Since we are considering the case $n = 2d$, and since $\Mhat_{\P,\Q,d} \subseteq \Mhat_{\P,d}$ by definition, it follows from Theorem~\ref{thm intro single portrait} that the dimension of $\Mhat_{\P,\Q,d}$ is bounded above by $\min\{d-1,n-2\} = d - 1$. Table~\ref{tab:intro dim_deg2_and3} (which we also include as Table~\ref{tab:dim_deg2_and3} in Section~\ref{sec:computational} for convenience) lists, for $d = 2, 3$ and $-1 \le n \le d - 1$, the number of pairs of admissible degree-$d$ portraits on $2d$ points up to combinatorial equivalence for which, in characteristic zero, we have $\dim \wh{\M}_{\P,\Q,d} = n$. (Here and throughout the article, we use the convention that the empty variety has dimension $-1$.)

\begin{table}[h]
    \caption{For each $d \in \{2,3\}$ and $n \in \{-1,0,1,2\}$, the number of equivalence classes of pairs $\{\P,\Q\}$ of admissible degree-$d$ portraits on $[2d]$ satisfying $\dim \Mhat_{\P,\Q,d} = n$}
    \label{tab:intro dim_deg2_and3}
    \begin{tabular}{l||cccc||c}
    & \multicolumn{4}{c||}{$\dim \Mhat_{\P,\Q,d}$} & \\\hline
    $d$ & $-1$ & $0$ & $1$ & $2$ & Total\\\hline
    $2$ & $198$ & $568$ & $14$ &  & $780$\\
    $3$ & $52310$ & $1297349$ & $1065$ & $18$ & $1350742$
    \end{tabular}
\end{table}

While the dimension of $\Mhat_{\P,\Q,d}$ typically agrees with our expectation of $0$, there are portrait pairs $\{\P,\Q\}$ achieving the full range of possible dimensions for $\Mhat_{\P,\Q,d}$.

Theorem \ref{thm intro 2 image} identifies an interesting family of extreme examples of portraits on $n = 2d$ points 
for which $\M_{\P,\Q,d}$ is either as small as possible (i.e., $\M_{\P,\Q,d} = \emptyset$) or as large as possible (i.e., $(d - 1)$-dimensional).
Theorem \ref{thm intro 2 image} accounts for all but one of the $14$ (equivalence classes of) portrait pairs $\{\P,\Q\}$ for which $\dim \wh{\M}_{\P,\Q,2} = 1$ and all of the $18$ pairs for which $\dim \wh{\M}_{\P,\Q,3} = 2$.

First, some terminology: The \emph{fiber partition} of a portrait $\P$ on $n$ points is the partition of $[n]$ given by $\Pi_\P := \{\P^{-1}(i) : i \in \P([n])\}$. A \emph{two-image portrait} is a portrait $\P: [2d] \rightarrow [2d]$ such that $\Pi_\P$ consists of two sets with $d$ elements each. 

\begin{thm}
\label{thm intro 2 image}
Let $d\geq 1$ and let $\Pi := \{A,B\}$ be a partition of $[2d]$ into two sets with $d$ elements each.
\begin{enumerate}
    \item If $\P$ and $\Q$ are two-image portraits with $\Pi_\P = \Pi_\Q = \Pi$, then
    \[
        \conf_{\P,d} = \conf_{\Q,d}.
    \]
    Hence $\conf_{\P,d}$ depends only on the partition $\Pi$; let $\conf_{\Pi,d} := \conf_{\P,d}$ and $\M_{\Pi,d} := \conf_{\Pi,d}/\aff_1$. Thus
    \[
        \M_{\P,\Q,d} = \M_{\Pi,d}.
    \]
    
    \item $\conf_{\Pi,d}$ is the $(d+1)$-dimensional Zariski closed subspace of $\conf^{2d}$ defined by the equations
    \[
        e_k(x_A) = e_k(x_B),
    \]
    for $1\leq k < d$, where $e_k$ is the $k$th elementary symmetric function in $d$ variables and $x_A := \{x_a : a \in A\}$ is the subset of $\{x_1, x_2, \ldots, x_{2d}\}$ of variables indexed by $A$, likewise for $x_B$. 
    
    \item If $\Pi' \neq \Pi$ is any other partition of $[2d]$ into two sets of $d$ elements each and $K$ is a field of characteristic 0, then
    \[
        \conf_{\Pi,d}(K) \cap \conf_{\Pi',d}(K) = \emptyset.
    \]
    Equivalently, if $\P$ and $\Q$ are two-image portraits with distinct fiber partitions, then
    \[
        \M_{\P,\Q,d}(K) = \emptyset.
    \]
    
    \item For each $q \in \conf_{\Pi,d}$, $\endo(q)$ contains at least $2d(2d - 1)$ degree-$d$ endomorphisms of $q$. Hence $E_{2d,d} \geq 2d(2d - 1)$.
\end{enumerate}
\end{thm}

We deduce Theorem \ref{thm intro 2 image}(3) as a corollary of the following result of possible independent interest.

\begin{thm}
\label{thm:intro URS}
Let $K$ be a field of characteristic 0, and suppose that $f(x), g(x) \in K[x]$ are polynomials such that
\[
    f^{-1}(\{0,1\}) = g^{-1}(\{0,1\})
\]
as sets with multiplicity. Then either $f(x) = g(x)$ or $f(x) = 1 - g(x)$.
\end{thm}

A close analysis of portrait pairs $\{\P,\Q\}$ for which the dimension of $\wh{\M}_{\P,\Q,d}$ differs from expectation revealed the following result as the most common source of deviations.

\begin{thm}
\label{thm intro left associates}
Let $\P$ and $\Q$ be admissible degree-$d$ portraits on $n$ points and suppose that that either
\begin{enumerate}
    \item $d = 2$ and there is at least one pair $i, j \in [n]$ such that $\P(i) = \P(j)$ and $\Q(i) = \Q(j)$,
    \item $d = 3$ and there are at least two pairs $i, j \in [n]$ such that $\P(i) = \P(j)$ and $\Q(i) = \Q(j)$, or
    \item $d$ is arbitrary and the fiber partitions $\Pi_\P$, $\Pi_\Q$ have a common fiber with $d$ elements.
\end{enumerate}
If $q$ is a degree-at-most-$d$ realization of both $\P$ and $\Q$ via the polynomials $f$ and $g$, respectively, then $f$ and $g$ are left associates.
\end{thm}

The condition $f(x) = \ell(g(x))$ for any pair of realizations $f(x), g(x)$ is highly restrictive. In Section \ref{sec:obstructions} we identify a number of ways in which Theorem \ref{thm intro left associates} can be used to explain why $\wh{\M}_{\P,\Q,d} = \emptyset$ for $\P$ and $\Q$ satisfying one of the conditions of the theorem. Furthermore, in Section \ref{sec:conditional dimension} we show how Theorem \ref{thm intro left associates} leads to a conditional revised expected dimension for $\wh{\M}_{\P,\Q,d}$ which heuristically accounts for the majority of the deviations towards higher dimensions.

\subsubsection{Endomorphisms of symmetric configurations}

Our survey of $\wh{\M}_{\P,\Q,2}$ for portraits on 4 points allows us to answer Question \ref{quest:intro max cardinality} for $(n,d) = (4,2)$ in characteristic 0: the maximal cardinality of $\endo_2(q)$ with $q \in \conf^4(\CC)$ is 
$E_{4,2} = 28$
and is realized by the highly symmetric configuration $q = \mu_4 = (1,i,-1,-i)$ of 4th roots of unity. Note that this more than doubles the lower bound of $E_{4,2} \geq 12$ given by Theorem \ref{thm intro 2 image}. This suggests that the sequence of configurations $q = \mu_n$ of $n$th roots of unity may be good candidates for understanding the extremal behavior of $\endo(q)$. Table~\ref{tab:rou endos intro} lists the number of degree-$d$ endomorphisms of $\mu_n$ for $3 \leq n \leq 8$ and $0 \leq d < n$.

\begin{table}[h]
    \caption{The cardinality of $\endo_d(\mu_n)$ for $3 \le n \le d$ and $0 \le d \le n - 1$}
    \label{tab:rou endos intro}
    \centering
    \begin{tabular}{c||ccccccccc}
    \diagbox[height=20pt]{$n$}{$d$} & $0$ & $1$ & $2$  & $3$   & $4$    & $5$     & $6$      & $7$\\\hline
    $3$ & $3$ & $3$ & $21$ &       &        &         &          &\\
    $4$ & $4$ & $4$ & $28$ & $220$ &        &         &          &\\
    $5$ & $5$ & $5$ & $5$  & $105$ & $3005$ &         &          &\\
    $6$ & $6$ & $6$ & $6$  & $30$  & $1992$ & $44616$ &          &\\
    $7$ & $7$ & $7$ & $7$  & $7$   & $105$  & $4907$  & $818503$ &\\
    $8$ & $8$ & $8$ & $8$  & $8$   & $280$  & $2968$  & $186840$ & $16587096$
    \end{tabular}
\end{table}

Note that for each degree $d$ the $n$ monomials $\zeta_n^k x^d$ belong to $\endo_d(\mu_n)$. In Section~\ref{sec:endos} we prove that, as suggested by the table, these are the only endomorphisms of $\mu_n$ with degree less than $n/2$. For convenience, we also reproduce this table in Section~\ref{sec:endos} as Table~\ref{tab:rou endos}.

\begin{thm}
\label{thm intro rou endos}
If $n > 2d \geq 1$, then the only degree-$d$ polynomial endomorphisms of $\mu_n$ in $\CC[x]$ are of the form $\zeta_n^k x^d$ for some $k\geq 0$.
\end{thm}

Theorem \ref{thm intro rou endos} is an immediate consequence of a result of Cargo and Schneider \cite[Theorem 1]{cargo/schneider}; however, we provide a self-contained proof in Section \ref{sec:endos}.

\subsection{Related work}

Questions \ref{quest:intro semigroup} and \ref{quest:intro max cardinality} relate to several conjectures on preperiodic points in arithmetic dynamics. We note two examples below.

\subsubsection{Uniform boundedness conjecture}\label{sec:ubc}

The Morton-Silverman uniform boundedness conjecture \cite[p. 100]{morton/silverman} asserts (in part) that for any number field $K$ of degree $e$ over $\QQ$ and any polynomial $f(x) \in K[x]$ with degree $d$, the number of $K$-rational preperiodic points of $f(x)$ is uniformly bounded by a constant $B_{d,e}$ depending only on $d$ and $e$. See \cite[\textsection 3.3]{Silverman_ADS} for the precise statement of the uniform boundedness conjecture and a list of further references.

The uniform boundedness conjecture is equivalent to the statement that $\endo_d(q) = \emptyset$ for any configuration $q$ consisting of more than $B_{d,e}$ points in $\AA^1(K)$. We now briefly justify this equivalence, which relies on the fact that if $q_1,\ldots,q_n$ is the full set of $K$-rational preperiodic points for a polynomial $f(x) \in K[x]$ of degree $d$, then for $q = (q_1,\ldots,q_n)$ we have $f \in \endo_d(q)$. In fact, from this observation it follows immediately that if $\endo_d(q) = \emptyset$ for all configurations $q$ of more than $B_{d,e}$ points in $\AA^1(K)$, then no degree-$d$ polynomial can have more than $B_{d,e}$ $K$-rational preperiodic points.

Now suppose the uniform boundedness conjecture holds, and let $q$ be a configuration of points in $\AA^1(K)$ with more than $B_{d,e}$ points. Since $B_{d,e}$ depends on $d$, we may as well assume\footnote{In fact, if $B_{d,e}$ exists, it is not too difficult to show that the inequality $B_{d,e} \ge d$ must necessarily hold. For the purposes of this discussion, though, we do not need to prove this.} that $B_{d,e} \ge d$. Suppose for contradiction that there exists $f(x) \in \endo_d(q)$. Then $q$ is a set of preperiodic points for $f$, so it remains to show that $f$ has coefficients in $K$, which would then contradict the uniform boundedness conjecture. But the fact that $f$ has coefficients in $K$ follows immediately from Lagrange interpolation (see also the proof of Proposition~\ref{prop portrait space is a variety}) since $q$ has more than $d = \deg f$ points.

\subsubsection{Common preperiodic point conjecture}
DeMarco, Krieger, and Ye \cite{DKY} conjecture that for any degree $d$ there exists a uniform bound $C_d$ such that if $f(x), g(x) \in \CC[x]$ are polynomials of degree $d$, then either
\[
    |\preper(f) \cap \preper(g)| \leq C_d \quad\text{or}\quad
    \preper(f) = \preper(g),
\]
where $\preper(f)$ denotes the set of all complex preperiodic points for $f$. Hence, if $q$ is a configuration of more than $C_d$ complex points, then this conjecture implies that all the elements of $\endo_d(q)$ have identical sets of preperiodic points.

\subsubsection{Dynamical moduli spaces}
The portrait moduli spaces $\M_{\P,d}$ are closely related, but not identical, to dynamical moduli spaces that have been studied since at least the 1980s. Typically one begins with a family $\F$ of endomorphisms of the projective line $\PP^1$ (e.g. degree-$d$ rational functions, degree-$d$ polynomials, or degree-$d$ polynomials with a single critical point,) and a portrait $\P$ on $[n]$, then constructs the space
\[
    \F[\P] := \{(f, q_1,\ldots,q_n) : f \in \F,\ f(q_i) = q_{\P(i)} \text{ for all } i, \text{ and } q_i \ne q_j \text{ for } i \ne j\}.
\]
If $G \subseteq \aut(\PP^1) \cong \PGL_2$ is a subgroup stabilizing $\F$ under conjugation, then $\phi \in G$ acts on $\F[\P]$ by
\[
    (f,q_1,\ldots, q_n)^\phi = (\phi\circ f\circ \phi^{-1}, \phi(q_1), \ldots, \phi(q_n)).
\]
Then we get a moduli space of dynamical systems of type $\F$ with \emph{level structure} $\P$ by taking the quotient $\M[\P] := \F[\P]/G$.

Our spaces $\conf_{\P,d}$ and $\M_{\P,d}$ are the projections of $\F[\P]$ and $\M[\P]$, respectively, onto the $q$-coordinates, where we have taken $\F$ to be the family of degree-$d$ polynomials and $G = \aff_1$ the group of affine linear transformations. The spaces $\conf_{\P,d}$ are more natural than $\F[\P]$ in the context of Questions \ref{quest:intro semigroup} and \ref{quest:intro max cardinality}, as configurations $q$ with exceptional endomorphism semigroups arise as points of intersection of $\conf_{\P,d}$ for several $\P$ and $d$. Furthermore, when $n > d$, Lagrange interpolation implies that the degree-$d$ polynomial $f(x)$ witnessing $q \in \conf_{\P,d}$ is uniquely determined by $q$, hence no information is lost in the projection $\pi: \F[\P] \rightarrow \conf_{\P,d}$ (see Theorem~\ref{thm intro single portrait}).

There is a large and growing literature on dynamical moduli spaces. We recommend the article \cite{doyle/silverman}, which constructs the parameter spaces $\endo_d^N[\P]$ and moduli spaces $\M_d^N[\P] := \endo_d^N[\P]/\PGL_{N+1}$, where $\endo_d^N$ denotes the family of all degree-$d$ endomorphisms of $\PP^N$, and develops the basic theory of such spaces. See \cite[\textsection 2]{doyle/silverman} for a brief survey of prior work on dynamical moduli spaces and \cite[Theorem 17.5]{doyle/silverman} for a statement showing that the general Morton-Silverman uniform boundedness conjecture can be rephrased in terms of $K$-rational points on dynamical moduli spaces.

\subsection{Organization}

The remainder of the paper is divided into three sections. The parameter and moduli spaces $\conf_{\P,d}$ and $\M_{\P,d}$ are constructed in Section \ref{sec:realization space}. Theorem \ref{thm intro single portrait} is proved as a combination of Proposition \ref{prop portrait space is a variety} and Corollary \ref{cor:Mpd_dim}. In Section \ref{sec:data} we study intersections of the portrait realization spaces and discuss the results of computational surveys in low degrees. Theorem \ref{thm intro 2 image} appears in parts as Theorem \ref{thm 2 image} and Proposition \ref{prop two image obstruction}; Theorem \ref{thm:intro URS} is proved as Theorem \ref{thm polynomial URS}; and Theorem \ref{thm intro left associates} is proved as Theorem \ref{thm left associate realizations}. Finally, in Section \ref{sec:endos} we prove that $E_{4,2} = 28$ (Proposition \ref{prop:4th_rou}) and prove Theorem \ref{thm intro rou endos} as Theorem \ref{thm rou endos}.

\subsection*{Acknowledgments}
This material is based upon work supported by NSF grant DMS-1439786 while the authors were in residence at the Institute for Computational and Experimental Research in Mathematics (ICERM) in Providence, RI, during the Summer@ICERM 2019 program on Computational Arithmetic Dynamics. We thank ICERM and Brown University for their hospitality during the eight-week program. We would like to thank Ben Hutz, Bianca Thompson, and Adam Towsley who, along with the second author, organized the Summer@ICERM program. We thank Grayson Jorgenson for helpful discussions while this article was being prepared, and we thank Paul Fili for lending us computing resources. John Doyle was partially supported by NSF grant DMS-2112697, Trevor Hyde was partially supported by the NSF Postdoctoral Research Fellowship DMS-2002176 and the Jump Trading Mathlab Research Fund, Colby Kelln was partially supported by an NSF Graduate Research Fellowship under grant DGE-1650441, and Max Weinreich was supported by an NSF Graduate Research Fellowship under grant DGE-2040433.

\section{Portrait realization spaces and moduli spaces}
\label{sec:realization space}

Let $\P$ be a portrait on $n$ points, that is, $\P : [n] \rightarrow [n]$ is a set-theoretic endomorphism of $[n] := \{1,2,\ldots, n\}$. Recall from the introduction that $(q_1,\ldots,q_n) \in \conf^n$ is a {\it degree-$d$ realization} of $\P$ if there exists a degree-$d$ polynomial $f(x)$ such that $f(q_i) = q_{\P(i)}$ for all $1 \le i \le n$. For the purposes of this paper, we consider the zero polynomial to have degree $0$.

\begin{example}\label{ex:portrait_example}
It is convenient to represent portraits on $[n]$ as directed graphs with vertices labeled $1,2,\ldots,n$. Figure~\ref{fig:portrait_example} illustrates three such diagrams: On the left is the portrait $\P$ that maps $1$, $2$, $3$, and $4$ to $1$, $1$, $2$, and $4$, respectively. In the middle is the portrait $\Q$ that maps $1$, $2$, $3$, and $4$ to $1$, $3$, $3$, and $1$, respectively. On the right, is the pair of portraits $\{\P,\Q\}$: $\P$ is drawn with solid red arrows and $\Q$ is drawn with dashed black arrows.

To conclude this example, note that $q = (q_1,q_2,q_3,q_4) = (0, 1, 2, 3)$ is a degree-$2$ realization of the pair of portraits $\{\P,\Q\}$. Indeed, if $f(x) := \frac{1}{2}x(x - 1)$ and $g(x) := -x(x - 3)$, then
    \begin{alignat*}{3}
    f(q_1) &= f(0) = 0 = q_1 \qquad\qquad g(q_1) &&= g(0) = 0 = q_1\\
    f(q_2) &= f(1) = 0 = q_1 \qquad\qquad g(q_2) &&= g(1) = 2 = q_3\\
    f(q_3) &= f(2) = 1 = q_2 \qquad\qquad g(q_3) &&= g(2) = 2 = q_3\\
    f(q_4) &= f(3) = 3 = q_4 \qquad\qquad g(q_4) &&= g(3) = 0 = q_1,
    \end{alignat*}
so $q$ realizes $\P$ and $\Q$ via the polynomials $f$ and $g$, respectively.
    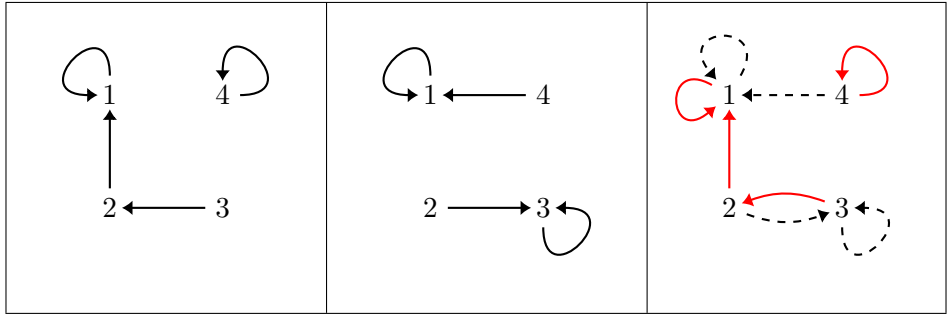
\begin{figure}[h]
    \centering
    \begin{tabular}{|c|c|c|}
    \hline
        \begin{tikzpicture}[shorten >= -2pt, scale=0.75]
        \node (0) at (0, 2) {$1$};
        \node (1) at (0, 0) {$2$};
        \node (2) at (2, 0) {$3$};
        \node (3) at (2, 2) {$4$};
        \draw[-{Latex[length=1.4mm,width=2mm]}, thick] (0) edge [out=90, in=180, looseness=7] (0);
        \draw[-{Latex[length=1.4mm,width=2mm]}, thick] (1) edge (0);
        \draw[-{Latex[length=1.4mm,width=2mm]}, thick] (2) edge (1);
        \draw[-{Latex[length=1.4mm,width=2mm]}, thick] (3) edge [out=0, in=90, looseness=7] (3);
        %
        \draw[-{Latex[length=1.4mm,width=2mm]}, thick, white] (2) edge [out=270, in=0, looseness=7] (2);
        \end{tikzpicture}
        &
        \begin{tikzpicture}[shorten >= -2pt, scale=0.75]
        \node (0) at (0, 2) {$1$};
        \node (1) at (0, 0) {$2$};
        \node (2) at (2, 0) {$3$};
        \node (3) at (2, 2) {$4$};
        \draw[-{Latex[length=1.4mm,width=2mm]}, thick] (0) edge [out=90, in=180, looseness=7] (0);
        \draw[-{Latex[length=1.4mm,width=2mm]}, thick] (1) edge (2);
        \draw[-{Latex[length=1.4mm,width=2mm]}, thick] (2) edge [out=270, in=0, looseness=7] (2);
        \draw[-{Latex[length=1.4mm,width=2mm]}, thick] (3) edge (0);
        \end{tikzpicture}
        &
        \begin{tikzpicture}[shorten >= -2pt, scale=0.75]
        \node (0) at (0, 2) {$1$};
        \node (1) at (0, 0) {$2$};
        \node (2) at (2, 0) {$3$};
        \node (3) at (2, 2) {$4$};
        \draw[-{Latex[length=1.4mm,width=2mm]}, thick, red] (0) edge [out=150, in=220, looseness=6] (0);
        \draw[-{Latex[length=1.4mm,width=2mm]}, thick, red] (1) edge (0);
        \draw[-{Latex[length=1.4mm,width=2mm]}, thick, red] (2) edge [out=160, in=20] (1);
        \draw[-{Latex[length=1.4mm,width=2mm]}, thick, red] (3) edge [out=0, in=90, looseness=7] (3);
        \draw[-{Latex[length=1.4mm,width=2mm]}, thick, dashed] (0) edge [out=60, in=130, looseness=6] (0);
        \draw[-{Latex[length=1.4mm,width=2mm]}, thick, dashed] (1) edge [out=340, in=200] (2);
        \draw[-{Latex[length=1.4mm,width=2mm]}, thick, dashed] (2) edge [out=270, in=0, looseness=7] (2);
        \draw[-{Latex[length=1.4mm,width=2mm]}, thick, dashed] (3) edge (0);
        \end{tikzpicture}\\
    \hline
    \end{tabular}
    \caption{An illustration of $\P$, $\Q$, and $\{\P,\Q\}$, respectively.}
    \label{fig:portrait_example}
\end{figure}
\end{example}

We are interested in the \emph{portrait realization spaces}
\begin{align*}
    \conf_{\P,d} &:= \{q \in \conf^n : \text{$q$ is a degree-$d$ realization of $\P$}\},\\
    \wh{\conf}_{\P,d} &:= \bigcup_{0 \le e \le d} \conf_{\P,e}
        = \{q \in \conf^n : \text{$q$ is a degree-at-most-$d$ realization of $\P$}\}.
\end{align*}
In this section, we study portrait realization spaces, culminating in a proof of Theorem~\ref{thm intro single portrait}.

\subsection{\texorpdfstring{$\conf_{\P,d}$}{Conf{P,d}} as a variety } 
The following lemma is used in the proofs of Propositions \ref{prop portrait space is a variety} and \ref{prop admissible}.

\begin{lemma}
\label{lem:vandermonde}
Let $(q_1, q_2,\ldots,q_n) \in \conf^n$ be a configuration of $n$ distinct points, and let $d \ge n$. Then for any portrait $\P: [n] \rightarrow [n]$, there exists a monic polynomial $f(x)$ of degree $d$ such that $f(q_i) = q_{\P(i)}$ for all $1 \leq i \leq n$.
\end{lemma}

\begin{proof}
The polynomial
$
    f(x) = x^d + a_{d-1}x^{d-1} + \ldots + a_1x + a_0
$
satisfies the conditions $f(q_i) = q_{\P(i)}$ for all $i = 1,\ldots,n$ if and only if $(a_{d-1},\ldots,a_1,a_0)$ is a solution to the system of equations
    \[
    q_i^d + a_{d-1}q_i^{d-1} + \cdots + a_1q_1 + a_0 = q_{\P(i)} \quad\text{for $i = 1,\ldots,n$}.
    \]
This can be rewritten as
    \[
        \begin{pmatrix}
            q_1^{d-1} & \cdots & q_1 & 1\\
            \vdots &  & \vdots & \vdots\\
            q_n^{d-1} & \cdots & q_n & 1
        \end{pmatrix}
        \begin{pmatrix}
            a_{d-1}\\
            \vdots\\
            a_0
        \end{pmatrix}
        =
        \begin{pmatrix}
            q_{\P(1)} - q_1^d\\
            \vdots\\
            q_{\P(n)} - q_n^d
        \end{pmatrix}.
    \]
The matrix on the left-hand side is an $n \times d$ Vandermonde matrix, hence has full rank since $q_1,\ldots,q_n$ are distinct. Since we assumed $d \ge n$, it follows that this system (which has rank $n$) has solutions, and therefore there exists a monic degree-$d$ polynomial $f$ satisfying the desired conditions.
\end{proof}

\begin{prop}
\label{prop portrait space is a variety}
$\confhat_{\P,d}$ is a Zariski closed subset of $\conf^n$, and $\conf_{\P,d}$ is a Zariski open subset of $\confhat_{\P,d}$ for any portrait $\P:[n] \rightarrow [n]$ be a portrait and any $d\geq 0$.
\end{prop}

\begin{proof}
If $d \ge n$, then it follows from Lemma~\ref{lem:vandermonde} that for any configuration $q \in \conf^n$, there exists a {\it monic} degree-$d$ polynomial $f(x)$ such that $f(q_i) = q_{\P(i)}$ for all $i \in [n]$. Therefore, in this case
\[
    \conf_{\P,d} = \confhat_{\P,d} = \conf^n,
\]
so the conclusion holds trivially.

Now, suppose $d \le n - 1$.
The configuration space $\conf^n$ is, by definition, the complement of the hyperplane arrangement $\{x_i = x_j : i \neq j\}$ in $\AA^n$. Let $R_n := \ZZ[x_i, (x_j - x_k)^{-1} : 1 \leq i, j, k \leq n, j \neq k]$ be the coordinate ring of $\conf^n$. Lagrange interpolation tells us that
\[
    f(x) := \sum_{i=1}^n x_{\P(i)}\prod_{\substack{j = 1\\ j \neq i}}^n \frac{x - x_j}{x_i - x_j} = b_{n-1}x^{n-1} + b_{n-2}x^{n-2} + \cdots + b_1 x + b_0 \in R_n[x]
\]
is the unique polynomial of degree at most $n - 1$ with coefficients in $R_n$ which realizes the portrait $\P$ on the indeterminates $x_i$. Thus, a configuration $q = (q_1, q_2, \ldots, q_n)$ is a degree-at-most-$d$ realization of $\P$ if and only if $b_{n-1}(q) = b_{n-2}(q) = \cdots = b_{d+1}(q) = 0$. In other words, $\confhat_{\P,d}$ is the Zariski closed subset of $\conf^n$ defined by the vanishing of $b_k$ with $d < k \leq n - 1$. Moreover, the space $\conf_{\P,d}$ of degree-$d$ realizations of $\P$ is the open subset of $\confhat_{\P,d}$ defined by $b_d(q) \ne 0$.
\end{proof}

If $\P: [n] \rightarrow [n]$ is a portrait and $\sigma$ is a permutation of $[n]$, then we let $\P^\sigma := \sigma^{-1} \circ \P \circ \sigma$ denote the conjugate of $\P$ by $\sigma$, which amounts to the relabelling of $\P$ induced by $\sigma$. We end this section by showing that relabelled portraits have isomorphic realization spaces. 

\begin{prop}\label{prop:conjugation_isomorphism}
Let $\P : [n] \to [n]$ be a portrait and let $\sigma$ be a permutation of $[n]$. The map
    \begin{equation}\label{eq:perm_isom}
        \begin{split}
        \Phi_\sigma : \conf^n &\longrightarrow \conf^n\\
        (q_1,\ldots,q_n) &\longmapsto (q_{\sigma(1)},\ldots,q_{\sigma(n)})
        \end{split}
    \end{equation}
induces isomorphisms $\conf_{\P,d} \overset{\sim}{\longrightarrow} \conf_{\P^\sigma,d}$ and $\confhat_{\P,d} \overset{\sim}{\longrightarrow} \confhat_{\P^\sigma,d}$ for all $d \ge 0$.
\end{prop}

\begin{proof}
The map $\Phi_\sigma$ is certainly an automorphism of $\conf^n$ (with inverse $\Phi_{\sigma^{-1}}$), so it suffices to show that for all $d \ge 0$, the image of $\conf_{\P,d}$ under $\Phi_\sigma$ is equal to $\conf_{\P^\sigma}$. Let $q = (q_1,\ldots,q_n) \in \conf_{\P,d}$, and let $q' = (q_1',\ldots,q_n') := \Phi_\sigma(q)$. Since $q \in \conf_{\P,d}$, there exists a degree-$d$ polynomial $f(x)$ such that $f(q_i) = q_{\P(i)}$ for all $i = 1,\ldots,n$. Applying $f$ to the components of $q'$ yields
\[
    f(q_i')
        = f(q_{\sigma(i)})
        = q_{\P(\sigma(i))}
        = q_{\sigma(\P^\sigma(i))}
        = q_{\P^\sigma(i)}',
\]
hence $\Phi_\sigma(q) = q' \in \conf_{\P^\sigma,d}$, as we wished to show. The reverse inclusion follows by symmetry, applying the above argument to $\Phi_\sigma^{-1} = \Phi_{\sigma^{-1}}$.
\end{proof}

\subsection{Dimension of \texorpdfstring{$\conf_{\P,d}$}{Conf{P,d}}}

We now state necessary and sufficient conditions on a portrait $\P$ for the degree-$d$ realization space $\conf_{\P,d}$ to be nonempty and, in that case, compute the dimension of $\conf_{\P,d}$. \emph{Throughout this subsection, we assume that all realization spaces and moduli spaces are defined over a field $K$ of characteristic zero.}

The most dynamically interesting situation is when $d \geq 2$, but for completeness we also briefly discuss the cases $d = 0$ and $d = 1$. Certainly for a portrait $\P$ on $[n]$ to have degree-$0$ realizations, $\P$ must be a constant portrait, and conversely any configuration in $\conf^n$ can realize a constant portrait. We now record Proposition \ref{prop:lower_degrees}, which is a consequence of the following elementary facts about degree-$1$ polynomials $\ell(x)$:

\renewcommand{\theenumi}{\roman{enumi}}
    \begin{enumerate}
        \item $\ell$ is determined by how it acts on any two points;
        \item If $\ell(x) \ne x$, then $\ell$ has at most one fixed point; and
        \item If $\ell$ has a periodic point of exact period $k > 1$, then $\ell^k(x) = x$ and $\ell^m(x) \ne x$ for all $1 \le m < k$. (In particular, every point except the unique fixed point has period equal to $k$.)
    \end{enumerate}
\renewcommand{\theenumi}{\arabic{enumi}}

\begin{prop}
\label{prop:lower_degrees}
Let $n\geq 1$ be an integer and let $\P : [n] \rightarrow [n]$ be a portrait.
\begin{enumerate}
    \item $\conf_{\P,0} \neq \emptyset$ if and only if $\P$ is constant. In this case, $\conf_{\P,0} = \conf^n$ and $\dim\conf_{\P,0} = n$.
    
    \item $\conf_{\P,1} \neq \emptyset$ if and only if the following three conditions hold,
    \begin{enumerate}
        \item $\P$ is a bijection,
        \item $\P$ is either the identity function or has at most one fixed point, and
        \item If $\P$ has a $k$-cycle for some $k > 1$, then all but at most one point in $[n]$ belongs to a $k$-cycle.
    \end{enumerate}
    In this case, if $\gamma(\P)$ denotes the number of orbits of $\P$, then $\dim \conf_{\P,1}$ is equal to $\gamma(\P)$ when $\P$ has a fixed point and is equal to $\gamma(\P) + 1$ otherwise.
\end{enumerate}
\end{prop}

Now suppose that $d \geq 2$. There are two natural combinatorial obstructions to $\P$ admitting degree-$d$ realizations: First, if $\P$ is to be realized by a polynomial of degree $d$, then no element of $[n]$ should have more than $d$ preimages under $\P$. Second, for a given integer $k \ge 1$, a polynomial of degree $d \ge 2$ can have at most $M_k(d)$ periodic cycles of length $k$, where
\[
    M_k(x) := \frac{1}{k}\sum_{j \mid k} \mu(k/j)x^{j}
\]
is the {\it $k$th necklace polynomial}; this follows from the fact that a point of period $k$ for a degree-$d$ polynomial $f$ is a root of the {\it $k$th dynatomic polynomial}
\[
    \Phi_k(x) := \prod_{j \mid k} \big(f^j(x) - x\big)^{\mu(k/j)},
\]
which has degree $kM_k(d)$, so that $f$ has at most $M_k(d)$ cycles of length $k$. Here $f^j$ denotes the $j$th iterate $f \circ \cdots \circ f$ of $f$.

\begin{defn}
\label{defn:admissible}
A portrait $\P: [n] \rightarrow [n]$ is \emph{admissible in degree $d$} for $d \geq 2$ if
\begin{enumerate}
    \item Every element of $[n]$ has at most $d$ preimages under $\P$, and
    \item For every integer $k \ge 1$, $\P$ has at most $M_k(d)$ periodic cycles of length $k$.
\end{enumerate}
\end{defn}

Our discussion above implies that admissibility in degree $d$ is a necessary condition for $\conf_{\P,d}$ to be nonempty. We prove that admissibility is also sufficient; one should compare this to \cite[Theorem 14.2]{doyle/silverman}.

\begin{lemma}\label{lem:admissible}
Let $d \ge 2$, and let $\P$ be a degree-$d$ admissible portrait. If $K$ is a field of characteristic $0$, then $\conf_{\P,d}(\overline{K}) \ne \emptyset$.
\end{lemma}

\begin{proof}
It suffices to show that for each $d \ge 2$ there exists at least one degree-$d$ polynomial $f(x) \in \QQ[x]$ with the following properties:
\begin{enumerate}
    \item For all $k \ge 1$, $f$ has exactly $M_k(d)$ periodic cycles of length $k$, and
    \item Every preperiodic point for $f$ has exactly $d$ preimages in $\overline{\QQ}$; equivalently, $f$ has no preperiodic critical points.
\end{enumerate}
We claim that $f(x) = x^d + 1$ satisfies these properties. Indeed, $0$ is the only critical point for $f$, and it is straightforward to show that the orbit $\{0, f(0), f^2(0), \ldots\}$ is unbounded, hence $f$ has no preperiodic critical points. Thus, (2) holds.

Suppose for contradiction that (1) fails---that is, suppose $f$ has fewer than $M_k(d)$ periodic cycles of length $k$. Then the polynomial $g(x) := f^k(x) - x$ has fewer zeros than expected, hence $g$ and $g'$ have a common root $\alpha$. Since $g(\alpha) = f^k(\alpha) - \alpha = 0$, $\alpha$ is periodic with period dividing $k$, and $g'(\alpha) = (f^k)'(\alpha) = 1$. This means that $\alpha$ belongs to a \emph{rationally indifferent} cycle. By \cite[Theorem 9.3.2]{beardon}, this cycle must attract a critical point; since $0$ is the only critical point, the orbit of $0$ must approach the cycle containing $\alpha$, contradicting the fact that $0$ has orbit tending to $\infty$.
\end{proof}

\begin{prop}
\label{prop admissible}
Let $n \ge 1$ and $d \ge 2$ be integers, and let $\P: [n] \to [n]$ be a portrait. The following are equivalent in characteristic $0$:
    \begin{enumerate}
        \item $\P$ is admissible in degree $d$.
        \item $\conf_{\P,d} \ne \emptyset$.
        \item $\dim\conf_{\P,d} = \min\{d+1, n\}$.
    \end{enumerate}
\end{prop}

\begin{proof}
First, suppose $n \le d$. If $q = (q_1, q_2, \ldots,q_n) \in \conf^n$ is any configuration, then by Lemma \ref{lem:vandermonde} there exists a monic degree-$d$ polynomial $f(x)$ such that $f(q_i) = q_{\P(i)}$. As in the proof of Proposition~\ref{prop portrait space is a variety}, it follows that $\conf_{\P,d} = \conf^n$. Thus, when $n \le d$, (2) and (3) are true, and (1) is true since a portrait on at most $d$ points is automatically admissible in degree $d$.

Now suppose $n \ge d + 1$. 
Consider the augmented realization space
\[
    \Poly_{\P,d} := \{(f,q_1,\ldots,q_n) : \deg f = d \text{ and } f(q_i) = q_{\P(i)} \text{ for all } 1 \le i \le n\}.
\]
Thinking of a polynomial as a rational map for which $\infty$ is a totally ramified fixed point, the dimension counting results\footnote{Let $\P'$ be the {\it weighted} portrait, in the terminology of \cite{doyle/silverman}, obtained by adding to $\P$ a single fixed point of weight $d$, corresponding to a totally ramified fixed point. The cited results of \cite{doyle/silverman} show that the parameter space associated to $\P'$---denoted ${\rm End}_d^1[\P']$ in \cite{doyle/silverman}---has dimension $d + 2$, but we lose a dimension by moving the totally ramified fixed point to $\infty$.} of \cite[Theorem 15.8 and Remark 15.9]{doyle/silverman} imply that
\[
    \P \text{ is a degree-$d$ admissible portrait} \Longleftarrow \Poly_{\P,d} \ne\emptyset \iff \dim\Poly_{\P,d} = d + 1.
\]
Since we have already shown that (1) implies (2) in Lemma~\ref{lem:admissible}, it therefore suffices to show that $\Poly_{\P,d} \cong \conf_{\P,d}$. Consider the projection map
    \begin{align*}
        \Psi_d : \Poly_{\P,d} &\longrightarrow \conf_{\P,d}\\
            (f,q_1,\ldots,q_n) &\longmapsto (q_1,\ldots,q_n).
    \end{align*}
Certainly $\Psi_d$ is a morphism, and it is surjective by the definition of $\conf_{\P,d}$. Moreover, since $n \ge d + 1$, Lagrange interpolation implies that for each configuration $q \in \conf_{\P,d}$ there is a \emph{unique} polynomial $f$ for which $\Psi_d(f,q) = q$, so $\Psi_d$ is bijective as a map on sets. The coefficients of $f$ may be expressed as regular functions on $\conf^n$ (see the proof of Proposition~\ref{prop portrait space is a variety}) so the inverse $\Psi_d^{-1}$ is also a morphism. Therefore, $\Psi_d$ is an isomorphism, completing the proof.
\end{proof}

\subsection{Remarks on the moduli spaces \texorpdfstring{$\M_{\P,d}$}{M P,d} and \texorpdfstring{$\Mhat_{\P,d}$}{\hat{M} P,d}}
\label{sec:remark_on_moduli}
Let $\aff_1$ denote the group of linear polynomials with respect to composition.\footnote{From the perspective of $\PGL_2$ acting on $\wh{\CC}$ by M\"obius transformations, the group $\aff_1$ is the Borel subgroup of upper triangular matrices in $\PGL_2$.}
If $\ell(x) \in \aff_1$ and $f(x)$ is the degree-$d$ polynomial witnessing the realization $q \in \conf_{\P,d}$, then $\widetilde{f} := \ell \circ f \circ \ell^{-1}$ has degree $d$ and 
\[
    \widetilde{f}(\ell(q_i)) = \ell(f(q_i)) = \ell(q_{\P(i)}).
\]
Hence $\ell(q) = (\ell(q_1), \ell(q_2), \ldots, \ell(q_n)) \in \conf_{\P,d}$. Thus we define the \emph{portrait moduli spaces} $\M_{\P,d}$ and $\wh{\M}_{\P,d}$ as the quotients
\[
    \M_{\P,d} := \conf_{\P,d}/\aff_1
    \quad\text{and}\quad
    \wh{\M}_{\P,d} := \wh{\conf}_{\P,d}/\aff_1.
\] 
For $n\geq 2$, $\M_{\P,d}$ (resp., $\Mhat_{\P,d}$) is a fine moduli space for the moduli problem of degree-$d$ realizations (resp., degree-at-most-$d$ realizations) of $\P$. Indeed, $\M_{\P,d}$ and $\Mhat_{\P,d}$ are coarse moduli spaces by construction, hence it suffices to show that no point on $\confhat_{\P,d}$ (thus no point on $\conf_{\P,d}$) has a nontrivial stabilizer. Recall that $\aff_1$ acts sharply $2$-transitively on $\AA^1$. Thus, if $q = (q_1, q_2, \ldots, q_n) \in \confhat_{\P,d}$ and $\ell(x) \in \aff_1$ fixes $q_1$ and $q_2$, then $\ell(x) = x$.

Moreover, for $n\geq 2$,
$
    \dim \M_{\P,d} = \dim \conf_{\P,d} - 2,
$
since $\aff_1$ is a 2-dimensional algebraic group acting faithfully on $\conf_{\P,d}$. Thus Corollary \ref{cor:Mpd_dim} is an immediate consequence of Proposition \ref{prop admissible}.

\begin{cor}
\label{cor:Mpd_dim}
Let $n,d \ge 2$, and let $\P: [n]\to[n]$ be a portrait. The following are equivalent in characteristic zero:
\begin{enumerate}
    \item $\P$ is a degree-$d$ admissible portrait.
    \item $\M_{\P,d} \ne\emptyset$.
    \item $\dim \M_{\P,d} = \min\{d-1,n-2\}$.
\end{enumerate}
\end{cor}

While $\conf_{\P,d}$ is a Zariski open subset of $\confhat_{\P,d}$, it need not be dense in $\confhat_{\P,d}$. In fact, the dimension of $\conf_{\P,d}$ may be strictly less than that of $\confhat_{\P,d}$. We illustrate with two examples:
\begin{enumerate}
    \item If $\P$ is the portrait consisting of two $2$-cycles, then $\P$ is not degree-$2$ admissible, hence $\conf_{\P,2} = \emptyset$. On the other hand, $\confhat_{\P,2} = \conf_{\P,1}$ has dimension $3$ by Proposition~\ref{prop:lower_degrees}(2).
    \item If $\P$ is the portrait consisting of three $4$-cycles, then $\P$ is degree-$2$ admissible, so $\conf_{\P,2}$ is nonempty; in fact, the dimension of $\conf_{\P,2}$ is $3$ by Proposition~\ref{prop admissible}. On the other hand, by Proposition~\ref{prop:lower_degrees}(2), we have $\dim \conf_{\P,1} = 4$. Since $\confhat_{\P,2}$ is the (disjoint) union of $\conf_{\P,1}$ and $\conf_{\P,2}$, we have $\dim \confhat_{\P,2} = 4 > \dim \conf_{\P,2}$.
\end{enumerate}

On the other hand, if $n\leq 2d$ and $\P$ is not the identity portrait\footnote{If $\P$ is the identity portrait on $n$ elements, then $\dim_{\P,1} = n$ by Proposition~\ref{prop:lower_degrees}. Moreover, by Proposition~\ref{prop admissible}, $\P$ is admissible in degree $d$ if and only if $n \le d$. In this case, $\dim \conf_{\P,d} = d + 1 > n = \conf_{\P,1}$, so the subsequent discussion is still valid for the identity portrait $\P$.}, then $\dim \conf_{\P,1} \leq d + 1$ by Proposition \ref{prop:lower_degrees}, with equality when $\P$ consists of $d$ 2-cycles. Hence if $\P$ is admissible in degree $d$, then 
\[
    \dim \wh{\conf}_{\P,d} = \dim \conf_{\P,d} = d + 1.
\]
However, if $\dim \conf_{\P,1} = \dim \conf_{\P,d} = d + 1$, then the space $\conf_{\P,d}$ will still not be Zariski dense in $\confhat_{\P,d}$. Taking quotients, the above remarks are also valid for $\M_{\P,d} \subset \Mhat_{\P,d}$. In particular, if $n \leq 2d$ and $\P$ is not the identity portrait, then
\begin{equation}
\label{eqn dimension bound}
    \dim \wh{\M}_{\P,d} = \dim \M_{\P,d} = d - 1.
\end{equation}

Finally, note that Theorem \ref{thm intro single portrait} is equivalent to the combination of Proposition \ref{prop portrait space is a variety} and Corollary \ref{cor:Mpd_dim}.

\section{Intersections of realization spaces}
\label{sec:data}

In this section we study the intersections of portrait realization spaces. If $\P, \Q : [n] \rightarrow [n]$ are portraits, then $\conf_{\P,d}$ and $\conf_{\Q,d}$ live in the same ambient space $\conf^n$. Let
\[
    \M_{\P,\Q,d} := (\conf_{\P,d} \cap \conf_{\Q,d})/\aff_1
    \quad\text{and}\quad
    \wh{\M}_{\P,\Q,d} := (\wh{\conf}_{\P,d} \cap \wh{\conf}_{\Q,d})/\aff_1
\]
denote the moduli space of affine equivalence classes of configurations $q \in \conf^n$ which have degree $d$ (respectively, degree at most $d$) realizations of both $\P$ and $\Q$. For $(n,d) = (4,2)$ and $(6,3)$ we have conducted a computational survey of all the spaces $\wh{\M}_{\P,\Q,d}$ with $\P$ and $\Q$ admissible degree-$d$ portraits. An analysis of the findings and results inspired by the data are discussed below.

\subsection{Dimension heuristic}
\label{sec:dimension_heuristic}

Suppose for simplicity that $n > d \geq 1$. Let $q \in \conf^n$ be a configuration of $n$ points. If $q \in \conf_{\P,d}$, then the existence of a degree-$d$ endomorphism $f(x)$ of $q$ imposes $n - d - 1$ algebraic constraints on the coordinates of $q$. (See the proof of Proposition \ref{prop portrait space is a variety}.) Thus the expected dimension of $\conf_{\P,d} \cap \conf_{\Q,d}$ is
$
    n - 2(n - d - 1) = 2d - n + 2,
$
where we interpret a negative dimension to mean that the space is empty.

The 2-dimensional group $\aff_1$ acts freely on $\conf^n$, hence the expected dimension of the quotient $\M_{\P,\Q,d}$ is $2d - n$. The same dimension heuristic applies to $\wh{\M}_{\P,\Q,d}$. We emphasize that this is only a heuristic: the algebraic conditions imposed by the two portraits may not be independent, or the system of equations may only have degenerate solutions in $\AA^n \setminus \conf^n$. Since $\M_{\P,\Q,d} \subseteq \M_{\P,d}$, Corollary \ref{cor:Mpd_dim} implies that in characteristic zero we have $\dim \M_{\P,\Q,d} \leq d - 1$, and if we restrict to $n \leq 2d$, \eqref{eqn dimension bound} implies that $\dim \wh{\M}_{\P,\Q,d} \leq d - 1$.

\begin{example}
\label{ex:dim_count}
As noted above, if $n = 4$ and $d = 2$, then the expected dimension of $\M_{\P,\Q,2}$ is 0 and the maximal possible dimension is 1. Consider the portraits $\P$, $\Q_1$, $\Q_2$, and $\Q_3$ illustrated in
Figure~\ref{tab:different_dimensions_examples}. For $i = 1,2,3$, the spaces $\M_{\P,\Q_i,2}$ have dimensions $-1$, $0$ and $1$, respectively, in characteristic zero. More precisely, the space $\M_{\P,\Q_1,2} = \emptyset$, $\M_{\P,\Q_2,2} = \{(0,1,2,3)\}$ consists of exactly one point,
and $\M_{\P,\Q_3,2}$ is a genus 0 curve.

\begin{figure}[h]
    \centering
    \begin{tabular}{cccc}
    \begin{tikzpicture}[shorten >= 1pt, scale=0.5]
        \node[] () at (2, 4) {};
        \node[] () at (2, -1.5) {};
        \node (0) at (0, 0) {$1$};
        \node (1) at (2.5, 0) {$2$};
        \node (2) at (2.5, 2.5) {$4$};
        \node (3) at (0, 2.5) {$3$};
        \draw[-{Latex[length=1.5mm,width=2mm]}, thick] (0) to [out=245, in=295, looseness=7] (0);
        \draw[-{Latex[length=1.5mm,width=2mm]}, thick] (1) to [out=245, in=295, looseness=7] (1);
        \path[-{Latex[length=1.5mm,width=2mm]}, thick] (2) edge (0);
        \path[-{Latex[length=1.5mm,width=2mm]}, thick] (3) edge (1);
    \end{tikzpicture}
    &\hspace{10mm}
    \begin{tikzpicture}[shorten >= 1pt, scale=0.5]
        \node[] () at (2, 4) {};
        \node[] () at (2, -1.5) {};
        \node (0) at (0, 0) {$1$};
        \node (1) at (2.5, 0) {$2$};
        \node (2) at (2.5, 2.5) {$4$};
        \node (3) at (0, 2.5) {$3$};
        \draw[-{Latex[length=1.5mm,width=2mm]}, thick, red] (0) to [out=245, in=295, looseness=7] (0);
        \draw[-{Latex[length=1.5mm,width=2mm]}, thick, red] (1) to [out=245, in=295, looseness=7] (1);
        \path[-{Latex[length=1.5mm,width=2mm]}, thick, red] (2) edge (1);
        \path[-{Latex[length=1.5mm,width=2mm]}, thick, red] (3) edge (0);
    \end{tikzpicture}
    &\hspace{10mm}
    \begin{tikzpicture}[shorten >= 1pt, scale=0.5]
        \node[] () at (2, 4) {};
        \node[] () at (2, -1.5) {};
        \node (0) at (0, 0) {$1$};
        \node (1) at (2.5, 0) {$2$};
        \node (2) at (2.5, 2.5) {$4$};
        \node (3) at (0, 2.5) {$3$};
        \draw[-{Latex[length=1.5mm,width=2mm]}, thick, red] (0) to [out=245, in=295, looseness=7] (0);
        \path[-{Latex[length=1.5mm,width=2mm]}, thick, red] (1) edge (0);
        \draw[-{Latex[length=1.5mm,width=2mm]}, thick, red] (2) to
        [out=245, in=295, looseness=7] (2);
        \path[-{Latex[length=1.5mm,width=2mm]}, thick, red] (3) edge (1);
    \end{tikzpicture}
    &\hspace{10mm}
    \begin{tikzpicture}[shorten >= 1pt, scale=0.5]
        \node[] () at (2, 4) {};
        \node[] () at (2, -1.5) {};
        \node (0) at (0, 0) {$1$};
        \node (1) at (2.5, 0) {$2$};
        \node (2) at (2.5, 2.5) {$4$};
        \node (3) at (0, 2.5) {$3$};
        \draw[-{Latex[length=1.5mm,width=2mm]}, thick, red] (0) to [out=330, in=210] (1);
        \draw[-{Latex[length=1.5mm,width=2mm]}, thick, red] (1) to [out=150, in=30] (0);
        \path[-{Latex[length=1.5mm,width=2mm]}, thick, red] (2) edge (1);
        \path[-{Latex[length=1.5mm,width=2mm]}, thick, red] (3) edge (0);
    \end{tikzpicture}
    \\
    $\P$ &\hspace{10mm} $\Q_1$ &\hspace{10mm} $\Q_2$ &\hspace{10mm} $\Q_3$\\
    \end{tabular}
    \caption{}
    \label{tab:different_dimensions_examples}
\end{figure}

\end{example}

\subsection{Combinatorial equivalence}

Recall that if $\P: [n] \rightarrow [n]$ is a portrait and $\sigma$ is a permutation of $[n]$, then $\P^\sigma := \sigma^{-1} \circ \P \circ \sigma$ is the relabelling of $\P$ by $\sigma$.

\begin{defn}
\label{defn:equiv_pair}
Two unordered pairs $\{\P, \Q\}$ and $\{\P', \Q'\}$ of portraits on $[n]$ are \emph{combinatorially equivalent} if there exists a permutation $\sigma$ of $[n]$ such that $\{\P',\Q'\} = \{\P^\sigma, \Q^\sigma\}$. The combinatorial equivalence class of $\{\P,\Q\}$ is denoted $\langle \P,\Q\rangle$ .
\end{defn}

A combinatorial equivalence class of portraits $\{\P,\Q\}$ may be visualized as a pair of phase portraits acting on an unlabeled set; see Figure \ref{fig:equivalence_ex}.
Note that if $\{\P,\Q\}$ and $\{\P',\Q'\}$ are combinatorially equivalent, then applying Proposition~\ref{prop:conjugation_isomorphism} and taking intersections implies that $\M_{\P,\Q,d} \cong \M_{\P',\Q',d}$ and $\wh{\M}_{\P,\Q,d} \cong \wh{\M}_{\P',\Q',d}$.

\begin{figure}[H]
    \centering
    \begin{tikzpicture}[shorten >= 0pt, scale=0.4]
    \node (00) at (0, 0) {2};
    \node (10) at (4, 0) {4};
    \node (11) at (4, 4) {3};
    \node (01) at (0, 4) {1};
    \node (20) at (8, 0) {6};
    \node (21) at (8, 4) {5};
    \draw[-{Latex[length=1.5mm,width=2mm]}, thick, red] (00) edge (10);
    \draw[-{Latex[length=1.5mm,width=2mm]}, thick, red] (01) edge [out=-20, in=-160] (11);
    \draw[-{Latex[length=1.5mm,width=2mm]}, thick, red] (11) edge [out=250, in=110] (10);
    \draw[-{Latex[length=1.5mm,width=2mm]}, thick, red] (10) edge [out=-20, in=-160] (20);
    \draw[-{Latex[length=1.5mm,width=2mm]}, thick, red] (20) edge [out=70, in=290] (21);
    \draw[-{Latex[length=1.5mm,width=2mm]}, thick, red] (21) edge [out=0, in=90, looseness = 7] (21);

    \path[-{Latex[length=1.5mm,width=2mm]}, thick, dashed] (00) edge (01);
    \path[-{Latex[length=1.5mm,width=2mm]}, thick, dashed] (01) edge [out=20, in=160] (11);
    \path[-{Latex[length=1.5mm,width=2mm]}, thick, dashed] (10) edge [out=20, in=160] (20);
    \path[-{Latex[length=1.5mm,width=2mm]}, thick, dashed] (11) edge [out=290, in=70] (10);
    \path[-{Latex[length=1.5mm,width=2mm]}, thick, dashed] (20) edge [out=110, in=250] (21);
    \path[-{Latex[length=1.5mm,width=2mm]}, thick, dashed] (21) edge (11);
\end{tikzpicture}
\hspace{5mm}
\begin{tikzpicture}[shorten >= 0pt, scale=0.4]
    \node (00) at (0, 0) {1};
    \node (10) at (4, 0) {2};
    \node (11) at (4, 4) {5};
    \node (01) at (0, 4) {6};
    \node (20) at (8, 0) {3};
    \node (21) at (8, 4) {4};
    \draw[-{Latex[length=1.5mm,width=2mm]}, thick, red] (00) edge (10);
    \draw[-{Latex[length=1.5mm,width=2mm]}, thick, red] (01) edge [out=-20, in=-160] (11);
    \draw[-{Latex[length=1.5mm,width=2mm]}, thick, red] (11) edge [out=250, in=110] (10);
    \draw[-{Latex[length=1.5mm,width=2mm]}, thick, red] (10) edge [out=-20, in=-160] (20);
    \draw[-{Latex[length=1.5mm,width=2mm]}, thick, red] (20) edge [out=70, in=290] (21);
    \draw[-{Latex[length=1.5mm,width=2mm]}, thick, red] (21) edge [out=0, in=90, looseness = 7] (21);

    \path[-{Latex[length=1.5mm,width=2mm]}, thick, dashed] (00) edge (01);
    \path[-{Latex[length=1.5mm,width=2mm]}, thick, dashed] (01) edge [out=20, in=160] (11);
    \path[-{Latex[length=1.5mm,width=2mm]}, thick, dashed] (10) edge [out=20, in=160] (20);
    \path[-{Latex[length=1.5mm,width=2mm]}, thick, dashed] (11) edge [out=290, in=70] (10);
    \path[-{Latex[length=1.5mm,width=2mm]}, thick, dashed] (20) edge [out=110, in=250] (21);
    \path[-{Latex[length=1.5mm,width=2mm]}, thick, dashed] (21) edge (11);
\end{tikzpicture}
\hspace{5mm}
\begin{tikzpicture}[shorten >= 0pt, scale=0.4]
    \node (00) at (0, 0) {$\bullet$};
    \node (10) at (4, 0) {$\bullet$};
    \node (11) at (4, 4) {$\bullet$};
    \node (01) at (0, 4) {$\bullet$};
    \node (20) at (8, 0) {$\bullet$};
    \node (21) at (8, 4) {$\bullet$};
    \draw[-{Latex[length=1.5mm,width=2mm]}, thick, red] (00) edge (10);
    \draw[-{Latex[length=1.5mm,width=2mm]}, thick, red] (01) edge [out=-20, in=-160] (11);
    \draw[-{Latex[length=1.5mm,width=2mm]}, thick, red] (11) edge [out=250, in=110] (10);
    \draw[-{Latex[length=1.5mm,width=2mm]}, thick, red] (10) edge [out=-20, in=-160] (20);
    \draw[-{Latex[length=1.5mm,width=2mm]}, thick, red] (20) edge [out=70, in=290] (21);
    \draw[-{Latex[length=1.5mm,width=2mm]}, thick, red] (21) edge [out=0, in=90, looseness = 7] (21);

    \path[-{Latex[length=1.5mm,width=2mm]}, thick, dashed] (00) edge (01);
    \path[-{Latex[length=1.5mm,width=2mm]}, thick, dashed] (01) edge [out=20, in=160] (11);
    \path[-{Latex[length=1.5mm,width=2mm]}, thick, dashed] (10) edge [out=20, in=160] (20);
    \path[-{Latex[length=1.5mm,width=2mm]}, thick, dashed] (11) edge [out=290, in=70] (10);
    \path[-{Latex[length=1.5mm,width=2mm]}, thick, dashed] (20) edge [out=110, in=250] (21);
    \path[-{Latex[length=1.5mm,width=2mm]}, thick, dashed] (21) edge (11);
\end{tikzpicture}
    \caption{The first two diagrams illustrate distinct but equivalent portrait pairs, and the third diagram illustrates their equivalence class. }
    \label{fig:equivalence_ex}
\end{figure}
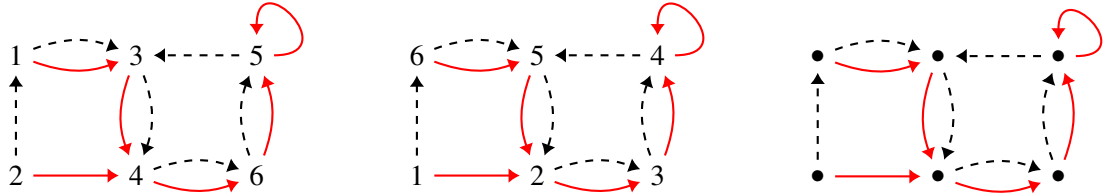

\subsection{Computational results}
\label{sec:computational}
We refer to a combinatorial equivalence class $\langle \P,\Q\rangle$ of admissible degree-2 (resp., degree-3) portraits on four (resp., six) points as a \emph{quadratic portrait pair} (resp., \emph{cubic portrait pair}). There are 780 quadratic portrait pairs and 1350742 cubic portrait pairs. For all such portrait pairs we computed basic invariants of the moduli spaces $\wh{\M}_{\P,\Q,d}$ of degree-at-most-$d$ realizations of $\P$ and $\Q$. \emph{Throughout this subsection, we assume that all realization spaces and moduli spaces are defined over a field $K$ of characteristic zero.}

The first invariant of $\wh{\M}_{\P,\Q,d}$ we consider is dimension. Note that
\[
    -1 \leq \dim \wh{\M}_{\P,\Q,d} \leq d - 1,
\]
where the upper bound holds because $n = 2d$ for $(n,d) = (4,2)$ or $(6,3)$ (see Section \ref{sec:dimension_heuristic}). Recall that we say the dimension of a space $\M$ is $-1$ when $\M = \emptyset$. The dimension heuristic derived in Section \ref{sec:dimension_heuristic} suggests that $\wh{\M}_{\P,\Q,d}$ should typically have dimension 0, hence that $\wh{\M}_{\P,\Q,d}$ should be a finite set. The dimensions of $\wh{\M}_{\P,\Q,d}$ for each quadratic and cubic portrait pair are tabulated in Table~\ref{tab:dim_deg2_and3} (labeled Table~\ref{tab:intro dim_deg2_and3} in the introduction).

\begin{table}[h]
    \caption{}
    \label{tab:dim_deg2_and3}
    \begin{tabular}{l||cccc|c}
    Dimension & $-1$ & $0$ & $1$ & $2$ & Total\\\hline\hline
    Quadratic pairs & $198$ & $568$ & $14$ &  & $780$\\\hline
    Cubic pairs & $52310$ & $1297349$ & $1065$ & $18$ & $1350742$
    \end{tabular}
\end{table}

In both degrees the most common dimension is 0, matching our expectation. However, there are many portrait pairs for which the dimension of $\wh{\M}_{\P,\Q,d}$ takes an unexpected value.

\begin{question}
\label{quest:dimension}
What combinatorial properties of $\langle \P, \Q\rangle$ imply that the dimension $\wh{\M}_{\P,\Q,d}$ will differ from the expected dimension?
\end{question}

\subsubsection{Dimension 0}\label{sec:0-dim}
Recall that the \emph{degree} of a $0$-dimensional scheme is the (finite) number of points on that scheme counted with multiplicity. For every quadratic portrait pair $\langle\P,\Q\rangle$ with $\dim \Mhat_{\P,\Q,2} = 0$, the degree of $\Mhat_{\P,\Q,2}$ was at most $7$, and in Table~\ref{tab:deg2_deg} we list the number of quadratic portrait pairs with $\dim \wh{\M}_{\P,\Q,2} = 0$ of each given degree.

\begin{table}[h]
    \centering
    \caption{}
    \label{tab:deg2_deg}
    \begin{tabular}{r||ccccccc|c}
        Degree & $1$ & $2$ & $3$ & $4$ & $5$ & $6$ & $7$ & Total \\\hline
        Quadratic pairs & $65$ & $166$ & $121$ & $116$ & $62$ & $29$ & $9$ & $568$
    \end{tabular}
\end{table}

The affine variety $\wh{\M}_{\P,\Q,d}$ is defined by polynomials with integral coefficients. Hence for each portrait pair $\langle \P, \Q \rangle$ with $\dim \wh{\M}_{\P,\Q,d} = 0$ and each prime field $K$, there is an associated splitting field for $\wh{\M}_{\P,\Q,d}$ over $K$. Of the 568 quadratic portrait pairs with 0-dimensional moduli spaces, the only pairs with any $\QQ$-rational realizations are the 65 with degree 1; note that since the spaces $\wh{\M}_{\P,\Q,d}$ are defined over $\ZZ$, degree 1 forces the unique point to be rational.

\begin{question}
If $\dim \wh{\M}_{\P,\Q,d} = 0$, how do properties of the splitting fields of $\wh{M}_{\P,\Q,d}$ reflect combinatorial properties of $\langle \P,\Q\rangle$?
\end{question}

Over 96\% of the cubic portrait pairs $\langle\P,\Q\rangle$ have $\dim \Mhat_{\P,\Q,3} = 0$. The histogram in Figure~\ref{fig:deg_deg3} summarizes the frequency with which each degree occurs for these portraits.

\begin{figure}[h]
    \includegraphics[scale=.5]{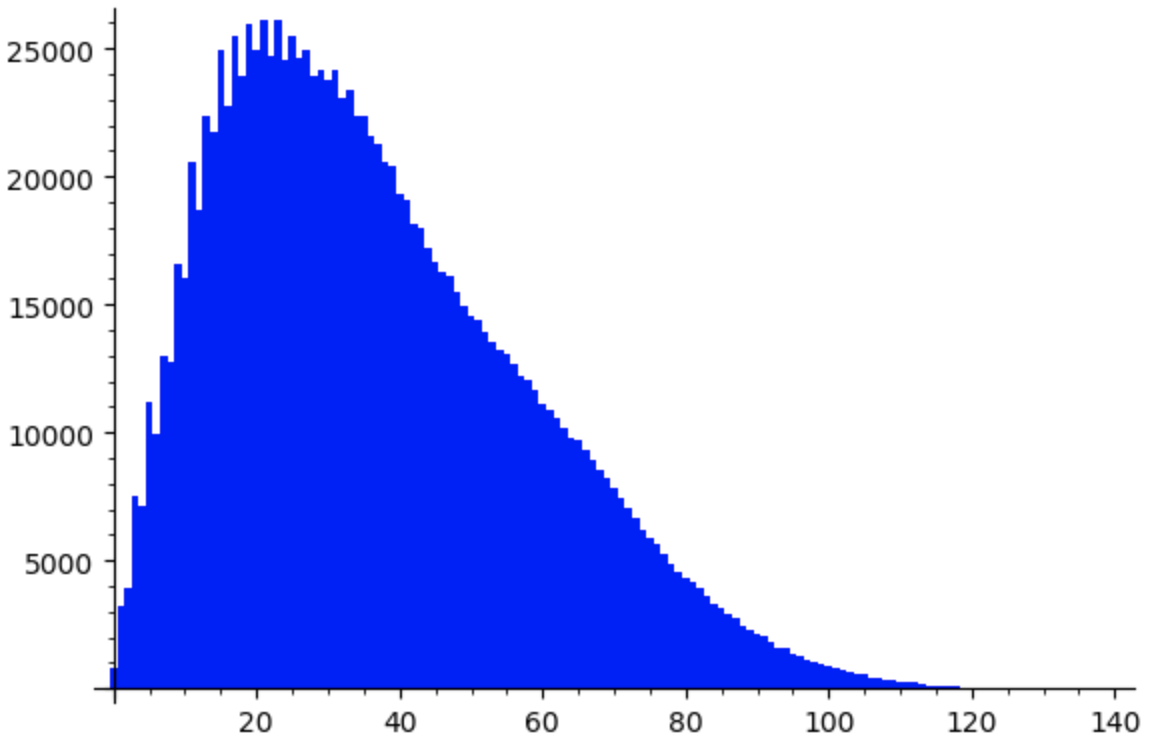}
    \caption{Degree of $\Mhat_{\P,\Q,3}$ for cubic pairs $\langle \P,\Q\rangle$ with $\dim \Mhat_{\P,\Q,3} = 0$. Degrees are listed along the horizontal axis; the vertical axis lists the number of cubic portrait pairs for which $\Mhat_{\P,\Q,3}$ has the given degree.}
    \label{fig:deg_deg3}
\end{figure}

The maximal degree is 144, achieved by the unique cubic portrait pair $\langle \P, \Q \rangle$ illustrated in Figure~\ref{fig:deg3_maxdeg}. We note that for this pair $\langle\P,\Q\rangle$, the space $\wh{\M}_{\P,\Q,3}$ is irreducible over $\QQ$, hence consists of 144 distinct, Galois conjugate points. The most common degree is 22, followed closely by 24, realized by 26083 and 26071 cubic portrait pairs respectively.

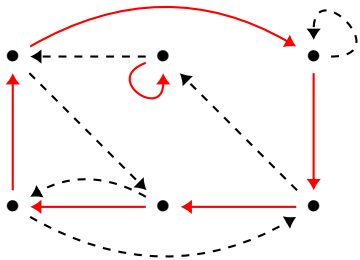
\begin{figure}[h]
    \centering
    \begin{tikzpicture}
        \node (1) at (0, 2) {$\bullet$};
        \node (0) at (2, 2) {$\bullet$};
        \node (2) at (4, 2) {$\bullet$};
        \node (5) at (0, 0) {$\bullet$};
        \node (4) at (2, 0) {$\bullet$};
        \node (3) at (4, 0) {$\bullet$};
        \draw[-{Latex[length=1.5mm,width=2mm]}, thick, red] (0) to [out=200, in=270, looseness=6] (0);
        \draw[-{Latex[length=1.5mm,width=2mm]}, thick, red] (1) to [out=30,in=150] (2);
        \draw[-{Latex[length=1.5mm,width=2mm]}, thick, red] (2) to (3);
        \draw[-{Latex[length=1.5mm,width=2mm]}, thick, red] (3) to (4);
        \draw[-{Latex[length=1.5mm,width=2mm]}, thick, red] (4) to (5);
        \draw[-{Latex[length=1.5mm,width=2mm]}, thick, red] (5) to (1);
        \draw[-{Latex[length=1.5mm,width=2mm]}, thick, dashed] (0) to (1);
        \draw[-{Latex[length=1.5mm,width=2mm]}, thick, dashed] (1) to (4);
        \draw[-{Latex[length=1.5mm,width=2mm]}, thick, dashed] (2) to [out=0,in=90,looseness=7] (2);
        \draw[-{Latex[length=1.5mm,width=2mm]}, thick, dashed] (3) to (0);
        \draw[-{Latex[length=1.5mm,width=2mm]}, thick, dashed] (4) to [out=150,in=30] (5);
        \draw[-{Latex[length=1.5mm,width=2mm]}, thick, dashed] (5) to [out=330,in=210] (3);
    \end{tikzpicture}
    \caption{The unique cubic portrait pair $\langle \P,\Q\rangle$ for which $\Mhat_{\P,\Q,3}$ is $0$-dimensional and achieves the maximal degree of $144$.}
    \label{fig:deg3_maxdeg}
\end{figure}

\begin{remark}
When $\Mhat_{\P,\Q,d}$ is zero-dimensional, one can typically apply B\'ezout's theorem to the defining equations for $\Mhat_{\P,d}$ and $\Mhat_{\Q,d}$ to get an upper bound on the degree of $\Mhat_{\P,\Q,d}$. (See the proof of Proposition~\ref{prop portrait space is a variety} for the equations, and recall that to construct $\Mhat$ from $\confhat$ one may set $q_1 = 0$ and $q_2 = 1$.) This bound depends on the portraits $\P$ and $\Q$, but in any case the upper bound coming from B\'ezout's theorem appears to be quite a bit higher than the degree of $\Mhat_{\P,\Q,d}$. One reason for this seems to be that a significant number of common ``realizations" $q$ of $\P$ and $\Q$ are degenerate, in the sense that $q_i = q_j$ for some pair $i \ne j$.
\end{remark}

We also observe an apparent bias towards even degrees in Figure \ref{fig:deg_deg3}, most easily seen in the interlaced spikes near and to the left of the mean. In fact, there are 28911 more cubic portrait pairs for which $\wh{\M}_{\P,\Q,3}$ has even degree than odd degree.
This bias also occurs in the quadratic case,
albeit less pronounced given the smaller data set.

\begin{question}
Is it the case for all degrees $d \geq 2$ that if $\wh{\M}_{\P,\Q,d}$ has dimension 0, then the degree of $\wh{\M}_{\P,\Q,d}$ is more likely to be even? If so, why does this bias occur?
\end{question}

\subsubsection{Maximal dimension}

There are 14 quadratic portrait pairs and 18 cubic portrait pairs $\langle \P, \Q\rangle$ for which $\Mhat_{\P,\Q,d}$ achieves the maximal possible dimension of $d - 1$. The 14 quadratic portrait pairs are illustrated in Figure \ref{fig:deg2_dim1}. All but one of these 14 quadratic portrait pairs and all 18 of these cubic portrait pairs are examples of \emph{two-image portraits with the same fiber partition} (see Section \ref{sec:two image}). In Theorem \ref{thm 2 image} we prove that $\wh{\M}_{\P,\Q,d}$ achieves the maximal possible dimension of $d - 1$ for all such portrait pairs.

\input{figs/table_deg2dim2.tex}

\begin{example}
\label{ex:exceptional_quad}
The bottom rightmost portrait in Figure \ref{fig:deg2_dim1} is the one example not explained by Theorem \ref{thm 2 image}. In this case, both portraits have the combinatorial type of the portrait $\P$ shown 
below:

\begin{center}
    \begin{tikzpicture}[shorten >= 0pt, scale=0.4]
    \node (00) at (0, 0) {$q_1$};
    \node (10) at (4, 0) {$q_2$};
    \node (20) at (8, 0) {$q_3$};
    \node (30) at (12, 0) {$q_4$};

    \draw[-{Latex[length=1.5mm,width=2mm]}, thick, red] (00) edge [out=180-45,in=180+45,looseness=7](00);
    \draw[-{Latex[length=1.5mm,width=2mm]}, thick, red] (10) edge (00);
    \draw[-{Latex[length=1.5mm,width=2mm]}, thick, red] (20) edge (10);
    \draw[-{Latex[length=1.5mm,width=2mm]}, thick, red] (30) edge [out=-45,in=45,looseness=7](30);
\end{tikzpicture}
    \vspace{-5mm}
\end{center}

The proof of Proposition~\ref{prop portrait space is a variety} implies that $\conf_{\P,2}$ is the hypersurface in $\conf^4$ defined by
\[
    q_2^2 - q_2q_3 + q_3^2 + q_1q_4 - q_1q_3 - q_2q_4 = 0.
\]

Observe that $q_i \mapsto q_{5 - i}$ is an automorphism of this hypersurface. Hence if $(q_1,q_2,q_3,q_4)$ is a realization of $\P$, then so is  $(q_4,q_3,q_2,q_1)$. Thus, if $\sigma$ is the permutation of $\{1,2,3,4\}$ defined by $i\mapsto 5-i$, then $\conf_{\P,2} = \conf_{\P^\sigma,2}$ and consequently $\M_{\P,\P^\sigma,2} = \M_{\P,2}$ has dimension 1. Note that $\P$ is not admissible in degree less than 2,
hence $\wh{\M}_{\P,2} = \M_{\P,2}$.
\end{example}

\subsubsection{Impossible portraits}
We say an equivalence class $\langle \P,\Q\rangle$ of  degree-$d$ admissible portraits is \emph{impossible in degree $d$} (resp., \emph{impossible in degree at most $d$}) if $\M_{\P,\Q,d} = \emptyset$ (resp., $\wh{\M}_{\P,\Q,d} = \emptyset$.)

\begin{example}
The quadratic portrait pair $\langle \P, \Q\rangle$ on the left of Figure~\ref{fig:impossible} is impossible in degree 2 but has realizations in degree 1; for example, $q = (1,i,-1,-i)$ realizes $\P$ and $\Q$ via the degree-$1$ polynomials $f(x) = ix$ and $g(x) = -ix$, respectively. The quadratic portrait pair on the right is impossible in degree at most 2.
\begin{figure}[h]
    \centering
    \begin{tikzpicture}[shorten >= 0pt, scale=0.4]
    \node (00) at (0, 0) {$\bullet$};
    \node (10) at (4, 0) {$\bullet$};
    \node (11) at (4, 4) {$\bullet$};
    \node (01) at (0, 4) {$\bullet$};

    \draw[-{Latex[length=1.5mm,width=1.5mm]}, thick, red] (00) edge [out=-20,in=180+20](10);
    \draw[-{Latex[length=1.5mm,width=1.5mm]}, thick, red] (10) edge [out=90-20,in=-90+20](11);
    \draw[-{Latex[length=1.5mm,width=1.5mm]}, thick, red] (11) edge [out=180-20,in=20](01);
    \draw[-{Latex[length=1.5mm,width=1.5mm]}, thick, red] (01) edge [out=-90-20,in=90+20](00);

    \path[-{Latex[length=1.5mm,width=1.5mm]}, thick, dashed] (00) edge [out=90-20,in=-90+20](01);
    \path[-{Latex[length=1.5mm,width=1.5mm]}, thick, dashed] (01) edge [out=-20,in=180+20](11);
    \path[-{Latex[length=1.5mm,width=1.5mm]}, thick, dashed] (11) edge [out=-90-20,in=90+20](10);
    \path[-{Latex[length=1.5mm,width=1.5mm]}, thick, dashed] (10) edge [out=180-20,in=20](00);
\end{tikzpicture}
\hspace{1in}
\begin{tikzpicture}[shorten >= 0pt, scale=0.4]
    \node (00) at (0, 0) {$\bullet$};
    \node (10) at (4, 0) {$\bullet$};
    \node (11) at (4, 4) {$\bullet$};
    \node (01) at (0, 4) {$\bullet$};

    \draw[-{Latex[length=1.5mm,width=1.5mm]}, thick, red] (00) edge [out=-20,in=180+20](10);
    \draw[-{Latex[length=1.5mm,width=1.5mm]}, thick, red] (10) edge [out=180-20,in=20](00);
    \draw[-{Latex[length=1.5mm,width=1.5mm]}, thick, red] (11) edge [out=-90+20,in=90-20](10);
    \draw[-{Latex[length=1.5mm,width=1.5mm]}, thick, red] (01) edge [out=-90-20,in=90+20](00);

    \path[-{Latex[length=1.5mm,width=1.5mm]}, thick, dashed] (00) edge (11);
    \path[-{Latex[length=1.5mm,width=1.5mm]}, thick, dashed] (01) edge [out=20,in=180-20](11);
    \path[-{Latex[length=1.5mm,width=1.5mm]}, thick, dashed] (11) edge [out=180+20,in=-20](01);
    \path[-{Latex[length=1.5mm,width=1.5mm]}, thick, dashed] (10) edge (01);
\end{tikzpicture}
    \caption{}
    \label{fig:impossible}
\end{figure}
\end{example}

There are 198 quadratic portrait pairs and 52310 cubic portrait pairs which are impossible in degree at most $2$ and $3$ respectively. We now turn to the following natural problem:

\begin{namedthm}{Obstruction Problem}
Determine whether a combinatorial class $\langle\P,\Q\rangle$ of admissible degree-$d$ portraits is impossible
based on the combinatorics of the portraits $\P$ and $\Q$.
\end{namedthm}

The simplest obstruction stems from Lagrange interpolation.

\begin{prop}[Interpolation Obstruction]
\label{prop interpolation obst}
Suppose that $\P$ and $\Q$ are admissible degree-$d$ portraits on $[n]$. If $\P(i) = \Q(i)$ for at least $d + 1$ elements $i \in [n]$ but $\P \neq \Q$, then $\langle \P, \Q\rangle$ is impossible in degree at most $d$.
\end{prop}

\begin{proof}
Lagrange interpolation implies that a degree-at-most-$d$ polynomial is uniquely determined by its values on $d + 1$ distinct points. If $f(x)$ and $g(x)$ are degree-at-most-$d$ polynomials realizing $\P$ and $\Q$, respectively, on a configuration $q$, then our assumption implies that $f(q_i) = g(q_i)$ for at least $d + 1$ distinct $i$, hence that $f(x) = g(x)$ as polynomials. But the assumption that $\P \neq \Q$ implies that $f(q_j) \neq g(q_j)$ for some $j$, a contradiction. Hence $\langle \P, \Q\rangle$ is impossible in degree at most $d$.
\end{proof}

There are 39 quadratic portrait pairs and 12140 cubic portrait pairs obstructed by Proposition~\ref{prop interpolation obst}; approximately 20\% and 24\% of the impossible portraits in degree 2 and 3, respectively.

\begin{example}
The following pair of admissible degree-3 portraits 
is obstructed by Proposition~\ref{prop interpolation obst} since the functions agree on $\{1,3,4,6\}$ but disagree on $\{2,5\}$.
    \begin{center}
    \begin{tikzpicture}[shorten >= 0pt, scale=0.4]
    \node (00) at (0, 0) {2};
    \node (10) at (4, 0) {4};
    \node (11) at (4, 4) {3};
    \node (01) at (0, 4) {1};
    \node (20) at (8, 0) {6};
    \node (21) at (8, 4) {5};
    \draw[-{Latex[length=1.5mm,width=2mm]}, thick, red] (00) edge (10);
    \draw[-{Latex[length=1.5mm,width=2mm]}, thick, red] (01) edge [out=-20, in=-160] (11);
    \draw[-{Latex[length=1.5mm,width=2mm]}, thick, red] (11) edge [out=250, in=110] (10);
    \draw[-{Latex[length=1.5mm,width=2mm]}, thick, red] (10) edge [out=-20, in=-160] (20);
    \draw[-{Latex[length=1.5mm,width=2mm]}, thick, red] (20) edge [out=70, in=290] (21);
    \draw[-{Latex[length=1.5mm,width=2mm]}, thick, red] (21) edge [out=0, in=90, looseness = 7] (21);

    \path[-{Latex[length=1.5mm,width=2mm]}, thick, dashed] (00) edge (01);
    \path[-{Latex[length=1.5mm,width=2mm]}, thick, dashed] (01) edge [out=20, in=160] (11);
    \path[-{Latex[length=1.5mm,width=2mm]}, thick, dashed] (10) edge [out=20, in=160] (20);
    \path[-{Latex[length=1.5mm,width=2mm]}, thick, dashed] (11) edge [out=290, in=70] (10);
    \path[-{Latex[length=1.5mm,width=2mm]}, thick, dashed] (20) edge [out=110, in=250] (21);
    \path[-{Latex[length=1.5mm,width=2mm]}, thick, dashed] (21) edge (11);
\end{tikzpicture}
    \end{center}
\end{example}

The obstruction problem is discussed further in Sections \ref{sec:left associate} and \ref{sec:two image}.

\subsection{Left associate realizations}
\label{sec:left associate}

In this section we introduce the restrictive notion of left associate polynomials, show how left associate realizations of portrait pairs may be detected combinatorially, and demonstrate their relation to deviations in the dimension of $\wh{\M}_{\P,\Q,d}$ from the generic expectation.

\begin{defn}
Polynomials $f(x)$ and  $g(x)$ are said to be \emph{left associates} if there exists a linear polynomial $\ell(x)$ such that $f(x) = \ell(g(x))$.
\end{defn}

Left association is a degree-respecting equivalence relation. One immediate and essential property of left associate polynomials is that they have the same fiber partition as self-maps of the affine line.

\begin{defn}
If $F: X \rightarrow X$ is an self-map of a set $X$, then the \emph{fiber partition} of $F$ is the set partition $\Pi_F$ of $X$ defined by 
\[
    \Pi_F := \{F^{-1}(x) : x \in X\}.
\]
\end{defn}

There are several ways to detect combinatorially when realizations of a portrait pair must be left associates. The most robust method, especially in low degree, is via common coincidence pairs. 

\begin{defn}
A \emph{coincidence pair} for a function $F: X \rightarrow Y$ is a pair $x_1, x_2 \in X$ of distinct elements such that $F(x_1) = F(x_2)$.
\end{defn}

We now restate Theorem~\ref{thm intro left associates} using this terminology.

\begin{thm}
\label{thm left associate realizations}
Let $\P$ and $\Q$ be admissible degree-$d$ portraits on $n$ points such that at least one of the following conditions holds:

\begin{enumerate}
    \item $d = 2$ and $\P$, $\Q$ have a common coincidence pair, 
    \item $d = 3$ and $\P$, $\Q$ have two common coincidence pairs, or
    \item $d$ is arbitrary and the fiber partitions $\Pi_\P$, $\Pi_\Q$ share a set with $d$ elements.
\end{enumerate}
If $q$ is a degree-at-most-$d$ realization of both $\P$ and $\Q$ via the polynomials $f$ and $g$, respectively, then $f$ and $g$ are left associates.
\end{thm}

\begin{remark}
Note that condition (1) in the statement of Theorem~\ref{thm left associate realizations} is a special case of condition (3); however, we state (1) separately to emphasize the theme of portraits having common coincidence pairs.
\end{remark}

We first establish a lemma which provides a general test for left associates. Given a polynomial $f(x)$, let $\delta f(x,y)$ be the two-variable symmetric polynomial defined by
\[
    \delta f(x,y) := \frac{f(x) - f(y)}{x - y}.
\]
For each integer $d\geq 1$, let $\rho_d : \AA^2 \rightarrow \AA^d$ be the function defined by
\[
    \rho_d(x,y) := \Big(\frac{x - y}{x - y}, \frac{x^2 - y^2}{x - y}, \ldots, \frac{x^d - y^d}{x - y}\Big).
\]

\begin{lemma}
\label{lemma left associates}
Let $f(x), g(x)$ be degree-at-most-$d$ polynomials. Suppose that $\{x_i, y_i\}$ for $1 \leq i < d$ are $d-1$ pairs of distinct points in $\AA^1$ such that
\begin{enumerate}
    \item $\delta f(x_i,y_i) = \delta g(x_i, y_i) = 0$ for each $i$, and
    \item the $d - 1$ vectors $\rho_d(x_i, y_i)$ with $1 \leq i < d$ are linearly independent.
\end{enumerate}
Then $f$ and $g$ are left associates.
\end{lemma}

\begin{proof}
If $f(x) = a_0 + a_1x + a_2x^2 + \ldots + a_dx^d$, then
\[
    \delta f(x,y) = \frac{f(x) - f(y)}{x - y}
    = a_1\frac{x - y}{x - y} + a_2 \frac{x^2 - y^2}{x - y} + \ldots + a_d\frac{x^d - y^d}{x - y}.
\]
Hence each point $(x_i, y_i)$ on the curve $\delta f(x,y) = 0$ imposes a linear condition on the vector of coefficients $(a_1, a_2, \ldots, a_d)$. From the definition of $\rho_d(x,y)$ and the assumed independence of $\rho_d(x_i, y_i)$ for $1 \leq i < d$, it follows that the points $(x_i, y_i)$ determine the vector $(a_1, a_2, \ldots, a_d)$ up to a scalar multiple. Thus if $g(x) = b_0 + b_1x + b_2x^2 + \ldots + b_d x^d$, then there is some nonzero scalar $c$ such that
\[
    (a_1, a_2, \ldots, a_d) = c(b_1, b_2, \ldots, b_d).
\]
Let $\ell(x) = cx + (a_0 - cb_0)$. Then $f(x) = \ell(g(x))$, as we wished to show.
\end{proof}

Now we turn to the proof of Theorem \ref{thm left associate realizations}.

\begin{proof}[Proof of Theorem \ref{thm left associate realizations}]
Suppose that $f(x)$ and $g(x)$ are degree-at-most-$d$ polynomials realizing $\langle \P, \Q\rangle$ on some configuration $q$. We show that in each of the three listed cases, $f(x)$ and $g(x)$ are left associates. As we remarked following the statement of Theorem~\ref{thm left associate realizations}, condition (1) is a special case of (3), so we omit the case of condition (1).

Assume that condition (2) holds. First, suppose that $q_1, q_2, q_3, q_4 \in \AA^1$ are points such that $\rho_3(q_1, q_2) = c \rho_3(q_3, q_4)$ for some scalar $c$. Comparing first components we see that $c = 1$, hence
\begin{align*}
    q_1 + q_2 &= q_3 + q_4\\
    q_1^2 + q_1q_2 + q_2^2 &= q_3^2 + q_3q_4 + q_4^2.
\end{align*}
Since $x^2 + xy + y^2 = (x + y)^2 - xy$, it follows that $q_1q_2 = q_3q_4$. The first two elementary symmetric functions of a pair of numbers uniquely determines the pair as a set, hence $\{q_1, q_2\} = \{q_3, q_4\}$. Therefore if $\{q_1, q_2\} \neq \{q_3, q_4\}$ are common coincidence pairs for cubic polynomials $f(x)$ and $g(x)$, then $\rho_3(q_1, q_2)$ and $\rho_3(q_3,q_4)$ are linearly independent and thus $f(x)$ is left associate to $g(x)$ by Lemma \ref{lemma left associates}.

We now assume that (3) holds. If $\Pi_\P$ and $\Pi_\Q$ share a part with $d$ elements, then $f^{-1}(q_i) = g^{-1}(q_j)$ for some $i$ and $j$. Thus $f(x) - q_i$ and $g(x) - q_j$ have the same roots with multiplicity. Unique factorization implies the existence of some $c\neq 0$ such that $f(x) - q_i = c(g(x) - q_j)$. If we let $\ell(x) := cx + (q_i - cq_j)$, then $f(x) = \ell(g(x))$.
\end{proof}

We apply Theorem \ref{thm left associate realizations} in Section \ref{sec:obstructions} to identify several obstructions to the realizations of quadratic and cubic portraits, and in Section \ref{sec:conditional dimension} we show how Theorem \ref{thm left associate realizations} leads to a conditional dimension heuristic that explains most of the unexpectedly large dimensions of $\wh{\M}_{\P,\Q,d}$ for degrees $d = 2, 3$.

\subsubsection{Obstructions}
\label{sec:obstructions}

Suppose $\{\P,\Q\}$ is a pair of admissible degree-$d$ portraits with distinct fiber partitions $\Pi_\P \neq \Pi_\Q$ and which satisfies any of the conditions in Theorem \ref{thm left associate realizations}. We claim that $\langle \P,\Q\rangle$ is an impossible portrait pair. Indeed, if $f(x)$, $g(x)$ is a pair of degree-at-most-$d$ polynomials realizing $\P$ and $\Q$ on some configuration $q$, then Theorem~\ref{thm left associate realizations} tells us that $f$ and $g$ are left associates, and it follows that $f(x)$ and $g(x)$ must have the same fiber partition. In particular, this implies that $\P$ and $\Q$ must have the same fiber partition, a contradiction. Hence $\langle \P, \Q \rangle$ is impossible in degree at most $d$. We call this the \emph{coincidence pair obstruction}.

The coincidence pair obstruction accounts for 67 (approximately 34\%) of the 198 quadratic portrait pairs which are impossible in degree at most 2 and accounts for 32573 (approximately 62\%) of the 52310 cubic portrait pairs which are impossible in degree at most 3. This is the most common known obstruction in degrees 2 and 3.

\begin{example}
The following cubic portrait pair $\langle\P,\Q\rangle$ illustrates the coincidence pair obstruction. Let $\P$ be the solid red portrait and let $\Q$ be the dashed black portrait.
\begin{center}
    \begin{tikzpicture}[shorten >= -2pt, scale=0.4]
    \node (00) at (0, 0) {2};
    \node (10) at (4, 0) {4};
    \node (11) at (4, 4) {3};
    \node (01) at (0, 4) {1};
    \node (20) at (8, 0) {6};
    \node (21) at (8, 4) {5};
    \draw[-{Latex[length=1.4mm,width=1.5mm]}, thick, dashed] (00) edge (10);
    \draw[-{Latex[length=1.4mm,width=1.5mm]}, thick, dashed] (01) edge [out=300, in=150] (10);
    \draw[-{Latex[length=1.4mm,width=1.5mm]}, thick, dashed] (11) edge (01);
    \draw[-{Latex[length=1.4mm,width=1.5mm]}, thick, dashed] (10) edge [out=120, in=-30] (01);
    \draw[-{Latex[length=1.4mm,width=1.5mm]}, thick, dashed] (20) edge (10);
    \draw[-{Latex[length=1.4mm,width=1.5mm]}, thick, dashed] (21) edge [out=90, in=0, looseness = 7] (21);

    \path[-{Latex[length=1.4mm,width=1.5mm]}, thick, red] (00) edge [out=270, in=180, looseness = 7] (00);
    \path[-{Latex[length=1.4mm,width=1.5mm]}, thick, red] (01) edge (00);
    \path[-{Latex[length=1.4mm,width=1.5mm]}, thick, red] (10) edge (21);
    \path[-{Latex[length=1.4mm,width=1.5mm]}, thick, red] (11) edge [out=-20, in=200] (21);
    \path[-{Latex[length=1.4mm,width=1.5mm]}, thick, red] (20) edge (21);
    \path[-{Latex[length=1.4mm,width=1.5mm]}, thick, red] (21) edge [out=160, in=20] (11);
\end{tikzpicture}
\end{center}
Then $\{1,2\}$ and $\{3,4\}$ are common coincidence pairs for $\P$ and $\Q$, but
\[
    \Pi_\P = \{\{1,2\}, \{3,4,6\}, \{5\}\} \neq \{\{1,2,6\}, \{3,4\},\{5\}\} = \Pi_\Q.
\]
Hence Theorem \ref{thm left associate realizations} implies that $\langle \P, \Q\rangle$ is impossible in degree at most $3$.
\end{example}

If we know that realizations of $\langle\P,\Q\rangle$ must be left associates, we may determine a partial portrait for the function $\L$ which realizes the left equivalence $\Q = \L \circ \P$. The portrait below illustrates this with $\P$ represented by the solid red portrait, $\Q$ represented by the dashed black portrait, and $\L$ represented by the dotted blue portrait. Note that $\L$ is only defined on the points in the image of $\P$.

\begin{center}
    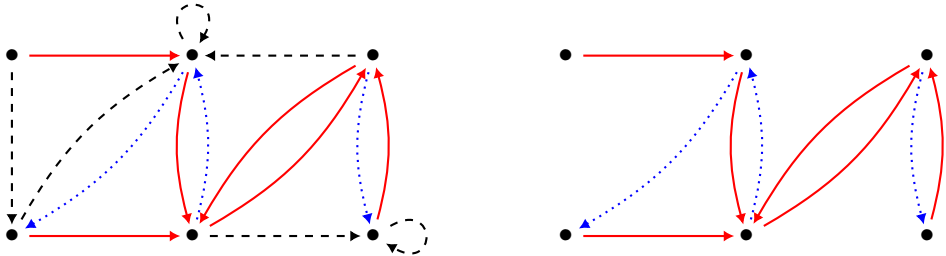
\begin{figure}[h]
\begin{tikzpicture}[shorten >= -2pt, scale=0.4]
    \node (00) at (0, 0){$\bullet$};
    \node (10) at (6, 0){$\bullet$};
    \node (11) at (6, 6){$\bullet$};
    \node (01) at (0, 6){$\bullet$};
    \node (20) at (12, 0){$\bullet$};
    \node (21) at (12, 6){$\bullet$};
    
    \draw[-{Latex[length=1.4mm,width=1.5mm]}, thick, dashed] (00) edge [out=45+15,in=45+180-15](11);
    \draw[-{Latex[length=1.4mm,width=1.5mm]}, thick, dashed] (01) edge (00);
    \draw[-{Latex[length=1.4mm,width=1.5mm]}, thick, dashed] (11) edge [out=90+30, in=90-30,looseness=7](11);
    \draw[-{Latex[length=1.4mm,width=1.5mm]}, thick, dashed] (10) edge (20);
    \draw[-{Latex[length=1.4mm,width=1.5mm]}, thick, dashed] (20) edge [out=30,in=-30,looseness=7](20);
    \draw[-{Latex[length=1.4mm,width=1.5mm]}, thick, dashed] (21) edge (11);

    \path[-{Latex[length=1.4mm,width=1.5mm]}, thick, red] (00) edge (10);
    \path[-{Latex[length=1.4mm,width=1.5mm]}, thick, red] (01) edge (11);
    \path[-{Latex[length=1.4mm,width=1.5mm]}, thick, red] (10) edge [out=45-15,in=45+180+15](21);
    \path[-{Latex[length=1.4mm,width=1.5mm]}, thick, red] (11) edge [out=-90-15,in=90+15](10);
    \path[-{Latex[length=1.4mm,width=1.5mm]}, thick, red] (20) edge [out=90-15,in=-90+15](21);
    \path[-{Latex[length=1.4mm,width=1.5mm]}, thick, red] (21) edge [out=45+180-15,in=45+15](10);
    
    \path[-{Latex[length=1.4mm,width=1.5mm]}, thick, dotted, blue] (11) edge [out=45+180+15,in=45-15] (00);
    \path[-{Latex[length=1.4mm,width=1.5mm]}, thick, dotted, blue] (10) edge [out=90-15,in=-90+15](11);
    \path[-{Latex[length=1.4mm,width=1.5mm]}, thick, dotted, blue] (21) edge[out=-90-15,in=90+15](20);
\end{tikzpicture}
\hspace{.5in}
\begin{tikzpicture}[shorten >= -2pt, scale=0.4]
    \node (00) at (0, 0){$\bullet$};
    \node (10) at (6, 0){$\bullet$};
    \node (11) at (6, 6){$\bullet$};
    \node (01) at (0, 6){$\bullet$};
    \node (20) at (12, 0){$\bullet$};
    \node (21) at (12, 6){$\bullet$};
    
    \draw[-{Latex[length=1.4mm,width=1.5mm]}, thick, white] (00) edge [out=45+15,in=45+180-15](11);
    \draw[-{Latex[length=1.4mm,width=1.5mm]}, thick, white] (01) edge (00);
    \draw[-{Latex[length=1.4mm,width=1.5mm]}, thick, white] (11) edge [out=90+30, in=90-30,looseness=7](11);
    \draw[-{Latex[length=1.4mm,width=1.5mm]}, thick, white] (10) edge (20);
    \draw[-{Latex[length=1.4mm,width=1.5mm]}, thick, white] (20) edge [out=30,in=-30,looseness=7](20);
    \draw[-{Latex[length=1.4mm,width=1.5mm]}, thick, white] (21) edge (11);

    \path[-{Latex[length=1.4mm,width=1.5mm]}, thick, red] (00) edge (10);
    \path[-{Latex[length=1.4mm,width=1.5mm]}, thick, red] (01) edge (11);
    \path[-{Latex[length=1.4mm,width=1.5mm]}, thick, red] (10) edge [out=45-15,in=45+180+15](21);
    \path[-{Latex[length=1.4mm,width=1.5mm]}, thick, red] (11) edge [out=-90-15,in=90+15](10);
    \path[-{Latex[length=1.4mm,width=1.5mm]}, thick, red] (20) edge [out=90-15,in=-90+15](21);
    \path[-{Latex[length=1.4mm,width=1.5mm]}, thick, red] (21) edge [out=45+180-15,in=45+15](10);
    
    \path[-{Latex[length=1.4mm,width=1.5mm]}, thick, dotted, blue] (11) edge [out=45+180+15,in=45-15] (00);
    \path[-{Latex[length=1.4mm,width=1.5mm]}, thick, dotted, blue] (10) edge [out=90-15,in=-90+15](11);
    \path[-{Latex[length=1.4mm,width=1.5mm]}, thick, dotted, blue] (21) edge[out=-90-15,in=90+15](20);
\end{tikzpicture}
\caption{On the left, a cubic portrait pair $\langle\P,\Q\rangle$ whose realizations must be left associates. On the right, a simplified representation of the first portrait, emphasizing $\P$ and $\L$ rather than $\P$ and $\Q$.}
\end{figure}
\end{center}

There are several more specialized obstructions we may identify in this setting coming from the fact that linear polynomials are determined by two values. We illustrate some of these obstructions in the following examples. In each case we consider an admissible degree-3 portrait $\P$ and a partial portrait $\L$ defined on the image of $\P$ and define $\Q := \L \circ \P$.

\begin{example}
If $\L$ has two fixed points, then any linear realization of $\L$ must be the identity $\ell(x) = x$. Hence if $\L$ fixes at least two but not all points, then $\langle \P, \Q\rangle$ is impossible in degree at most 3. Note that this occurs whenever realizations of $\P$ and $\Q$ are necessarily left associates and $\P(i) = \Q(i)$ for at least two values of $i$. For example, this obstructs the following pair $\{\P,\L\}$.
\begin{center}
    \begin{tikzpicture}[shorten >= -2pt, scale=0.4]
    \node (00) at (0, 0){$\bullet$};
    \node (10) at (4, 0){$\bullet$};
    \node (11) at (4, 4){$\bullet$};
    \node (01) at (0, 4){$\bullet$};
    \node (20) at (8, 0){$\bullet$};
    \node (21) at (8, 4){$\bullet$};

    \path[-{Latex[length=1.4mm,width=2mm]}, thick, red] (00) edge (01);
    \path[-{Latex[length=1.4mm,width=2mm]}, thick, red] (01) edge (11);
    \path[-{Latex[length=1.4mm,width=2mm]}, thick, red] (10) edge (11);
    \path[-{Latex[length=1.4mm,width=2mm]}, thick, red] (11) edge [out=90+30,in=90-30,looseness=7](11);
    \path[-{Latex[length=1.4mm,width=2mm]}, thick, red] (20) edge (10);
    \path[-{Latex[length=1.4mm,width=2mm]}, thick, red] (21) edge [out=90+30,in=90-30,looseness=7](21);
    
    \path[-{Latex[length=1.4mm,width=2mm]}, thick, dotted, blue] (11) edge (21);
    \path[-{Latex[length=1.4mm,width=2mm]}, thick, dotted, blue] (10) edge [out=-90+30,in=-90-30,looseness=7](10);
    \path[-{Latex[length=1.4mm,width=2mm]}, thick, dotted, blue] (01) edge [out=90+30,in=90-30,looseness=7](01);
    \path[-{Latex[length=1.4mm,width=2mm]}, thick, dotted, blue] (21) edge (20);
    \path[-{Latex[length=1.4mm,width=2mm]}, thick, dotted, blue] (21) edge (20);
\end{tikzpicture}
\end{center}
\end{example}

\begin{example}
A variation on the previous example: if $\L$ has a 2-cycle, then any linear realization $\ell(x)$ of $\L$ must be an involution. Thus if $\L$ is not an involution, then $\langle \P, \Q\rangle$ is impossible in degree at most 3, as in the following example.
\begin{center}
    \begin{tikzpicture}[shorten >= -2pt, scale=0.4]
    \node (00) at (0, 0){$\bullet$};
    \node (10) at (4, 0){$\bullet$};
    \node (11) at (4, 4){$\bullet$};
    \node (01) at (0, 4){$\bullet$};
    \node (20) at (8, 0){$\bullet$};
    \node (21) at (8, 4){$\bullet$};

    \path[-{Latex[length=1.4mm,width=1.5mm]}, thick, red] (00) edge (10);
    \path[-{Latex[length=1.4mm,width=1.5mm]}, thick, red] (01) edge (11);
    \path[-{Latex[length=1.4mm,width=1.5mm]}, thick, red] (10) edge (20);
    \path[-{Latex[length=1.4mm,width=1.5mm]}, thick, red] (11) edge (21);
    \path[-{Latex[length=1.4mm,width=1.5mm]}, thick, red] (20) edge (21);
    \path[-{Latex[length=1.4mm,width=1.5mm]}, thick, red] (21) edge [out=90+30,in=90-30,looseness=7](21);
    
    \path[-{Latex[length=1.4mm,width=1.5mm]}, thick, dotted, blue] (11) edge (10);
    \path[-{Latex[length=1.4mm,width=1.5mm]}, thick, dotted, blue] (10) edge (01);
    \path[-{Latex[length=1.4mm,width=1.5mm]}, thick, dotted, blue] (21) edge [out=-90+40,in=90-40](20);
    \path[-{Latex[length=1.4mm,width=1.5mm]}, thick, dotted, blue] (20) edge [out=90+40,in=-90-40](21);
\end{tikzpicture}
\end{center}
\end{example}

\begin{example}
Consider the following portrait pair $\{\P,\L\}$,
\begin{center}
    \begin{tikzpicture}[shorten >= -2pt, scale=0.4]
    \node (00) at (0, 0){2};
    \node (10) at (4, 0){4};
    \node (11) at (4, 4){3};
    \node (01) at (0, 4){1};
    \node (20) at (8, 0){6};
    \node (21) at (8, 4){5};

    \path[-{Latex[length=1.4mm,width=1.5mm]}, thick, red] (00) edge (11);
    \path[-{Latex[length=1.4mm,width=1.5mm]}, thick, red] (01) edge (11);
    \path[-{Latex[length=1.4mm,width=1.5mm]}, thick, red] (10) edge [out=90+15,in=-90-15](11);
    \path[-{Latex[length=1.4mm,width=1.5mm]}, thick, red] (11) edge [out=-90+15,in=90-15](10);
    \path[-{Latex[length=1.4mm,width=1.5mm]}, thick, red] (20) edge (10);
    \path[-{Latex[length=1.4mm,width=1.5mm]}, thick, red] (21) edge (20);
    
    \path[-{Latex[length=1.4mm,width=1.5mm]}, thick, dotted, blue] (11) edge [out=-45+15,in=-45+180-15](20);
    \path[-{Latex[length=1.4mm,width=1.5mm]}, thick, dotted, blue] (20) edge [out=-45+180+15,in=-45-15](11);
    \path[-{Latex[length=1.4mm,width=1.5mm]}, thick, dotted, blue] (00) edge [out=-15,in=180+15](10);
    \path[-{Latex[length=1.4mm,width=1.5mm]}, thick, dotted, blue] (10) edge [out=180-15,in=15](00);
    
\end{tikzpicture}
\end{center}
As discussed in the previous example, any linear realization of $\L$ must be a non-trivial involution. Notice that $\P(\L(i)) = \P(i)$ for the 4 points $i \in \{2,3,4,6\}$. If $f(x)$ and $\ell(x)$ are degree-at-most-3 and degree-1 realizations of $\P$ and $\L$, respectively, then the identity $f(\ell(x)) = f(x)$ holds identically by Lagrange interpolation. However, since $\ell(x)$ acts freely on a generic fiber of $f(x)$, the degree of $f(x)$ must be even. On the other hand, $\P$ is not admissible for any $d < 3$ since the point $3$ has three distinct preimages under $\P$. This contradiction implies that $\langle \P, \Q\rangle$ is impossible in degree at most 3.
\end{example}

\begin{example}
In this example, $\L$ has a 4-cycle.
\begin{center}
    \begin{tikzpicture}[shorten >= -2pt, scale=0.4]
    \node (10) at (4, 0){4};
    \node (11) at (4, 4){2};
    \node (01) at (0, 4){1};
    \node (20) at (8, 0){5};
    \node (21) at (8, 4){3};
    \node (30) at (12, 0){6};

    \path[-{Latex[length=1.4mm,width=1.5mm]}, thick, red] (01) edge (11);
    \path[-{Latex[length=1.4mm,width=1.5mm]}, thick, red] (10) edge (11);
    \path[-{Latex[length=1.4mm,width=1.5mm]}, thick, red] (11) edge (21);
    \path[-{Latex[length=1.4mm,width=1.5mm]}, thick, red] (21) edge (10);
    \path[-{Latex[length=1.4mm,width=1.5mm]}, thick, red] (20) edge (11);
    \path[-{Latex[length=1.4mm,width=1.5mm]}, thick, red] (30) edge (20);
    
    \path[-{Latex[length=1.4mm,width=1.5mm]}, thick, dotted, blue] (10) edge [out=90+30,in=-90-30](11);
    \path[-{Latex[length=1.4mm,width=1.5mm]}, thick, dotted, blue] (11) edge [out=30,in=180-30](21);
    \path[-{Latex[length=1.4mm,width=1.5mm]}, thick, dotted, blue] (21) edge [out=-90+30,in=90-30](20);
    \path[-{Latex[length=1.4mm,width=1.5mm]}, thick, dotted, blue] (20) edge [out=180+30,in=-30](10);
    
\end{tikzpicture}
\end{center}
If $\ell(x)$ is a linear realization of $\L$ over a field of characteristic different from $2$, then $\ell(x)$ is affine equivalent to $ix$ where $i^2 = -1$, and up to an affine change of coordinates we may assume that the 4-cycle of $\ell(x)$ is $\mu_4 := \{1, i, -1, -i\}$. The restriction of $\P$ to this 4-cycle happens to be the portrait for one of the exceptional quadratic endomorphisms of $\mu_4$ (see Proposition \ref{prop:4th_rou} below). Lagrange interpolation then implies that any degree-at-most-$3$ realization $f(x)$ of $\P$ must be quadratic. However, $\P$ is not an admissible degree-2 portrait since 2 has three pre-images under the map $\P$.
\end{example}

The subtlety of the obstructions in these examples suggests that it is unreasonable to expect a complete solution of the general obstruction problem. However, as we saw with the interpolation obstruction and the coincidence pair obstruction, there are some simple conditions that account for the majority of the impossible portraits in degrees 2 and 3.

\begin{question}
For $d \geq 4$, what are the most common obstructions to realizations of pairs of admissible degree-$d$ portraits on $[2d]$ points?
\end{question}

In Section \ref{sec:two image} we discuss one more obstruction related to the special family of two-image portraits.

\subsubsection{Conditional dimension heuristic}
\label{sec:conditional dimension}
Suppose that $\P$ and $\Q$ are admissible degree-$d$ portraits on $n > d$ points with the same fiber partition for which we know that that any pair of degree-at-most-$d$ polynomials realizing $\P$ and $\Q$, respectively, must be left associates. In this situation a different dimension heuristic for $\wh{\M}_{\P,\Q,d}$ applies. 

The data of a pair of degree-at-most-$d$ polynomials $f(x)$ and $g(x)$ realizing $\P$ and $\Q$ is, in this case, equivalent to the data of a degree-at-most-$d$ polynomial $f(x)$ and a linear polynomial $\ell(x)$ such that $g(x) = \ell(f(x))$. Let $m$ be the number of points in the image of $\P$, which is necessarily the same as that of $\Q$. Thus a point in $\wh{\M}_{\P,\Q,d}$ is specified by $n$ parameters subject to $n - d - 1$ constraints from interpolating $f(x)$ in degree $d$ and $m - 2$ constraints from interpolating $\ell(x)$ in degree 1, modulo the action of the two-dimensional group $\aff_1$. Hence the expected dimension of $\wh{\M}_{\P,\Q,d}$ is
\[
    n - (n - d - 1) - (m - 2) - 2 = d - m + 1.
\]

\begin{example}
Suppose that $\langle \P, \Q \rangle$ is a cubic portrait pair with at least 2 common coincidence pairs and the same fiber partition. 
Then Theorem \ref{thm left associate realizations} implies that any pair of realizations of $\langle \P, \Q \rangle$ must be left associates. Thus we expect the dimension of $\wh{\M}_{\P,\Q,3}$ to be $4 - m$. There are 6964 cubic portrait pairs satisfying these conditions; Table \ref{tab:conditional_dim} shows the number of such cubic portrait pairs with a given dimension and image size $m$. The bold entries denote the expected dimension.

\begin{table}[h]
    \caption{For $2 \le m \le 4$ and $-1 \le n \le 2$, the number of cubic portrait pairs $\langle \P,\Q\rangle$ with at least $2$ coincidence pairs and $m$ image points satisfying $\dim \Mhat_{\P,\Q,3} = n$.}
    \label{tab:conditional_dim}
    \begin{tabular}{c||cccc}
        & \multicolumn{4}{c}{$\dim \wh{\M}_{\P,\Q,3}$}\\\hline
        $m$    & $-1$   & $0$    & $1$   & $2$ \\\hline
        $2$ & $0$    & $0$    & $0$   & $\textbf{18}$\\
        $3$ & $98$   & $0$    & $\textbf{818}$ & $0$\\
        $4$ & $1536$ & $\textbf{4421}$ & $73$  & $0$
    \end{tabular}
\end{table}

Comparing Table \ref{tab:dim_deg2_and3} and Table \ref{tab:conditional_dim} we see that the conditional dimension heuristic accounts for the majority of cubic portrait pairs for which $\dim \wh{\M}_{\P,\Q,3} = 1$ and all of those for which $\dim \wh{\M}_{\P,\Q,3} = 2$.
\end{example}

\subsection{Two-image portraits}
\label{sec:two image}

In this section we study another exceptional family of portraits.

\begin{defn}
A \emph{two-image portrait} $\P$ is an admissible degree-$d$ portrait on $2d$ points such that $|\P([2d])| = 2$. That is, $\Pi_\P$ consists of two sets with $d$ elements each.
\end{defn}

There are three combinatorial types of two-image portraits $\P$ in each degree $d$, determined by their action on the points in the image of $\P$. Examples of these three types in the degree 3 case are illustrated below.

\begin{center}
    \begin{tikzpicture}[shorten >= -2pt, scale=0.3]
    \node (00) at (0, 0) {$\bullet$};
    \node (02) at (0, 4) {$\bullet$};
    \node (11) at (3, 2) {$\bullet$};
    \node (21) at (9, 2) {$\bullet$};
    \node (30) at (12, 0) {$\bullet$};
    \node (32) at (12, 4) {$\bullet$};
    \draw[-{Latex[length=1.4mm,width=1.5mm]}, thick, red] (00) edge (11);
    \draw[-{Latex[length=1.4mm,width=1.5mm]}, thick, red] (02) edge (11);
    \draw[-{Latex[length=1.4mm,width=1.5mm]}, thick, red] (11) edge [out=-45,in=45+180](21);
    
    \draw[-{Latex[length=1.4mm,width=1.5mm]}, thick, red] (30) edge (21);
    \draw[-{Latex[length=1.4mm,width=1.5mm]}, thick, red] (32) edge (21);
    \draw[-{Latex[length=1.4mm,width=1.5mm]}, thick, red] (21) edge [out=-45+180,in=45](11);
\end{tikzpicture}
\hspace{.2in}
\begin{tikzpicture}[shorten >= -2pt, scale=0.3]
    \node (00) at (0, 0) {$\bullet$};
    \node (02) at (0, 4) {$\bullet$};
    \node (11) at (3, 2) {$\bullet$};
    \node (21) at (9, 2) {$\bullet$};
    \node (30) at (12, 0) {$\bullet$};
    \node (32) at (12, 4) {$\bullet$};
    \draw[-{Latex[length=1.4mm,width=1.5mm]}, thick, red] (00) edge (11);
    \draw[-{Latex[length=1.4mm,width=1.5mm]}, thick, red] (02) edge (11);
    \draw[-{Latex[length=1.4mm,width=1.5mm]}, thick, red] (11) edge [out=-45,in=45,looseness=4](11);
    
    \draw[-{Latex[length=1.4mm,width=1.5mm]}, thick, red] (30) edge (21);
    \draw[-{Latex[length=1.4mm,width=1.5mm]}, thick, red] (32) edge (21);
    \draw[-{Latex[length=1.4mm,width=1.5mm]}, thick, red] (21) edge [out=-45+180,in=45+180,looseness=4](21);
\end{tikzpicture}
\hspace{.2in}
\begin{tikzpicture}[shorten >= -2pt, scale=0.3]
    \node (00) at (0, 0) {$\bullet$};
    \node (01) at (0, 2) {$\bullet$};
    \node (02) at (0, 4) {$\bullet$};
    \node (11) at (4, 2) {$\bullet$};
    \node (21) at (8, 2) {$\bullet$};
    \node (31) at (12, 2) {$\bullet$};
    \draw[-{Latex[length=1.4mm,width=1.5mm]}, thick, red] (00) edge (11);
    \draw[-{Latex[length=1.4mm,width=1.5mm]}, thick, red] (01) edge (11);
    \draw[-{Latex[length=1.4mm,width=1.5mm]}, thick, red] (02) edge (11);
    \draw[-{Latex[length=1.4mm,width=1.5mm]}, thick, red] (11) edge (21);
    
    \draw[-{Latex[length=1.4mm,width=1.5mm]}, thick, red] (31) edge (21);
    \draw[-{Latex[length=1.4mm,width=1.5mm]}, thick, red] (21) edge [out=45,in=90+45,looseness=5](21);
\end{tikzpicture}
\end{center}

Theorem \ref{thm 2 image} shows that the degree-$d$ realization space of a two-image portrait $\P$ depends only on the fiber partition $\Pi_\P$ of $\P$. As we will see, the special properties of \emph{two}-image portraits all stem from the fact that $\aff_1$ acts sharply \emph{two}-transitively on the affine line. Note that if $h$ is a symmetric function with integral coefficients and $q$ is a multiset of elements of a field $K$, then $h(q)$ has a well-defined value. 

\begin{thm}
\label{thm 2 image}
Let $d\geq 1$ and let $\Pi = \{A,B\}$ be a partition of $[2d]$ into two sets with $d$ elements.
\begin{enumerate}
    \item If $\P$ and $\Q$ are two-image portraits on $[2d]$ with the same fiber partition $\Pi_{\P} = \Pi_{\Q} = \Pi$, then
    \[
        \wh{\conf}_{\P,d} = \conf_{\P,d} = \conf_{\Q,d} = \wh{\conf}_{\Q,d}.
    \]
    Hence $\conf_{\P,d}$ depends only on the partition $\Pi$, and we let $\conf_{\Pi} := \conf_{\P,d} = \conf_{\Q,d}$ denote this common subspace of $\conf^{2d}$.
    
    \item $\conf_{\Pi}$ is the ($d + 1$)-dimensional Zariski closed subspace of $\conf^{2d}$ defined by the equations
    \[
        e_k(x_A) = e_k(x_B)
    \]
    for $1 \leq k < d$, where $e_k$ is the $k$th elementary symmetric function in $d$ variables, and $x_A, x_B$ are the subsets of the set of indeterminates $\{x_1, x_2, \ldots, x_{2d}\}$ indexed by the elements of $A$ and $B$ respectively.
    
    \item For each configuration $q \in \conf_\Pi$ there exist at least $2d(2d - 1)$ distinct degree-$d$ polynomials $f(x)$ such that $f(q) \subseteq q$.
\end{enumerate}
\end{thm}

\begin{proof}
(1) First note that since $\P$ contains a fiber with $d$ elements, $\P$ is not admissible for any degree less than $d$. Hence $\wh{\conf}_{\P,d} = \conf_{\P,d}$. Now suppose that $q \in \conf_{\P,d}$, and let $f(x)$ be the degree-$d$ polynomial realizing $\P$. Without loss of generality, suppose that $\im(\P) = \{1,2\}$ and $\im(\Q) = \{i,j\}$. Furthermore, since $\Pi_\P = \Pi_\Q = \{A,B\}$ we may suppose that $\P^{-1}(1) = \Q^{-1}(i) = A$ and $\P^{-1}(2) = \Q^{-1}(j) = B$. Let $\ell(x)$ be the unique linear polynomial such that $\ell(q_1) = q_i$ and $\ell(q_2) = q_j$. Then $g := \ell\circ f$ is a degree-$d$ polynomial such that for each $a \in A$,
\[
    g(q_a) = \ell(f(q_a))= \ell(q_1) = q_i,
\]
and similarly, $g(q_b) = q_j$ for each $b \in B$. Hence $g(x)$ realizes $\Q$ on the configuration $q$, which is to say that $q \in \conf_{\Q,d}$. Thus $\conf_{\P,d} \subseteq \conf_{\Q,d}$, and the reverse inclusion follows by symmetry.

(2) Let $\P : [2d] \to [2d]$ be a portrait with fiber partition $\Pi = \{A,B\}$, and let $i := \P(A)$ and $j := \P(B)$. By construction, $\P$ is admissible in degree $d$. If $q \in \conf_{\Pi} = \conf_{\P,d}$, and $f$ is a degree-$d$ polynomial realizing $\P$, then $f(x) - q_i$ and $f(x) - q_j$ vanish at the elements of $q_A$ and $q_B$ respectively. Comparing coefficients we conclude that $e_k(q_A) = e_k(q_B)$ for all $1 \leq k < d$. 

Conversely, if $q = (q_A, q_B)$ is a configuration such that $e_k(q_A) = e_k(q_B)$ for $1 \leq k < d$, then the polynomial
\[
    f(x) := x^d - e_1(q_A)x^{d-1} + e_2(q_A)x^{d-2} + \ldots + (-1)^{d-1}e_{d-1}(q_A)x
\]
is constant on each of $q_A$ and $q_B$. Let $a$ and $b$ be the images of $q_A$ and $q_B$ under $f$, respectively, and let $\ell$ be the unique linear polynomial such that $\ell(a) = q_i$ and $\ell(b) = q_j$. Then $\ell \circ f$ maps $q_A$ to $q_i$ and $q_B$ to $q_j$, thus $q$ is in $\conf_{\P,d} = \conf_\Pi$. Therefore, $\conf_{\Pi}$ is equal to the space defined by the equations $e_k(x_A) = e_k(x_B)$ for $1 \leq k < d$.

Given $q := (q_A, q_B) \in \conf_\Pi$, let $f(x) = b_d x^d + \ldots + b_1 x + b_0$ be the unique degree-$d$ polynomial such that $f(q_A) = 0$ and $f(q_B) = 1$. (The first condition determines the roots, hence determines the coefficients up to scaling; the second condition determines the leading coefficient.) Now consider the map $\iota: \conf_\Pi \rightarrow \AA^{d + 1}$ defined by $\iota(q) = (b_0, b_1,\ldots, b_d)$.
Lagrange interpolation implies that $\iota$ is an injective morphism;
moreover, $\iota$ maps surjectively onto the dense open subset of $\AA^{d+1}$ corresponding to all degree-$d$ polynomials for which $0$ and $1$ are not critical values. Therefore, $\dim \conf_\Pi = d + 1$.

(3) Given the partition $\Pi$, a two-image portrait $\P$ with fiber partition $\Pi$ is determined by choosing the ordered pair of values $(\P(A), \P(B))$. Hence there are $2d(2d - 1)$ such portraits. Part (1) shows that $\conf_{\P,d} = \conf_\Pi$ for any such portrait $\P$. Therefore, given a configuration $q \in \conf_\Pi$, there exist at least $2d(2d - 1)$ distinct degree-$d$ polynomials $f(x)$ such that $f(q) \subseteq q$.
\end{proof}

\begin{remark}
The proof of part (1) of Theorem~\ref{thm 2 image} generalizes directly to rational functions and three-image portraits on $3d$ points: For a portrait $\P : [n] \to [n]$, let $\conf_{\P,d}(\wh{K})$ denote the space of all configurations $q$ of $n$ distinct points on the projective line $\wh{K} := \PP^1(K)$ for which there is a degree-$d$ rational function $f$ realizing the portrait $\P$ on $q$. Then, if $\P, \Q$ are three-image portraits on $3d$ points with the same fiber partition, we have that
\[
    \conf_{\P,d}(\wh{K}) = \conf_{\Q,d}(\wh{K}).
\]
The main point is that the automorphism group of the affine line is sharply 2-transitive, whereas the automorphism group of the projective line is sharply 3-transitive.
\end{remark}

Continuing with the notation of Theorem \ref{thm 2 image}, let $\M_{\Pi,d}$ be the space defined by
\[
    \M_{\Pi,d} := \conf_{\Pi,d}/\aff_1.
\]
The following corollary is an immediate consequence of Theorem \ref{thm 2 image}.

\begin{cor}\label{cor:two-image}
Let $\P$ and $\Q$ be two-image portraits on $[2d]$ with the same fiber partition. Then
    \[
        \wh{\M}_{\P,\Q,d} = \M_{\P,\Q,d} = \M_{\Pi,d}
        \quad\text{and}\quad
    \dim \M_{\Pi,d} = d - 1.
    \]
\end{cor}

Since any two partitions of $[2d]$ into two $d$-element sets are related by a permutation, it follows from Proposition~\ref{prop:conjugation_isomorphism} and Corollary~\ref{cor:two-image} that the isomorphism class of $\M_{\Pi,d}$ depends only on the degree $d$ and not the partition $\Pi$.

\begin{example}
If $d = 2$ and $\Pi := \{\{1,2\},\{3,4\}\}$, then Theorem \ref{thm 2 image} (2) implies that
\[
    \M_{\Pi,2} = \left\{(0,1,q_3, q_4) \in \conf^4 : q_3 + q_4 = 1\right\},
\]
so $\M_{\Pi,2}$ is a line minus the points at which $0$, $1$, $q_3$, and $q_4$ are not pairwise distinct.
In particular, $\M_{\Pi,2}$ has infinitely many rational points. Of the 14 quadratic portrait pairs $\langle \P, \Q\rangle$ with $\dim \wh{\M}_{\P,\Q,2} = 1$, all but one is isomorphic to $\M_{\Pi,2}$ (see Example \ref{ex:exceptional_quad}). 
\end{example}

\begin{example}
If $d = 3$ and $\Pi := \{\{1,2,3\}, \{4,5,6\}\}$, then Theorem \ref{thm 2 image} (2) tells us that 
\begin{equation}\label{eq:P2ex}
    \M_{\Pi,3} = \left\{(0,1,q_3,q_4,q_5,q_6) \in \conf^6 : 1 + q_3 = q_4 + q_5 + q_6 \text{ and } q_3 = q_4q_5 + q_4q_6 + q_5q_6\right\}.
\end{equation}
Since \eqref{eq:P2ex} defines a quadric surface, $\M_{\Pi,3}$ is birational over $\overline{\QQ}$ to $\PP^2$. Furthermore, since this model of $\M_{\Pi,3}$ has a $\QQ$-rational point---take $(q_1, q_2, q_3, q_4, q_5, q_6) = (0, 1, -4, -2, 2, -3)$, for example---it follows that $\M_{\Pi,3}$ is birational over $\QQ$ to $\PP^2$. In particular, $\M_{\Pi,3}$ has infinitely many rational points.

For all 18 of the cubic portrait pairs $\langle \P, \Q\rangle$ with $\dim \wh{\M}_{\P,\Q,3} = 2$, the space $\Mhat_{\P,\Q,3}$ is isomorphic to the surface $\M_{\Pi,3}$.
\end{example}

Two-image portraits exhibit a remarkable rigidity: Theorem \ref{thm 2 image} tells us that two-image portraits $\P$ and $\Q$ with the same fiber partition have identical realization spaces. On the other hand, we now show that two-image portraits with different fiber partitions have disjoint realization spaces. 

\begin{prop}[Two-Image Obstruction]
\label{prop two image obstruction}
Suppose that $\P$ and $\Q$ are two-image portraits on $2d$ points. If $\Pi_\P \neq \Pi_\Q$, then, in characteristic zero, $\langle \P, \Q\rangle$ is impossible in degree at most $d$.
\end{prop}

\begin{remark}
Note that Theorem \ref{thm 2 image} and Proposition \ref{prop two image obstruction} combine to give Theorem \ref{thm intro 2 image}.
\end{remark}

The two-image obstruction is relatively rare: there are 24 quadratic portrait pairs and 91 cubic portrait pairs obstructed by Proposition \ref{prop two image obstruction}. We deduce Proposition \ref{prop two image obstruction} from the following result. 

\begin{thm}
\label{thm polynomial URS}
Let $K$ be a field of characteristic $0$, and suppose that $f(x), g(x) \in K[x]$ are polynomials such that
\[
    f^{-1}(\{0,1\}) = g^{-1}(\{0,1\})
\]
as sets with multiplicity. Then either $f(x) = g(x)$ or $f(x) = 1 - g(x)$.
\end{thm}

\begin{remark}
By two-transitivity of $\aff_1$, the set $\{0,1\}$ in Theorem \ref{thm polynomial URS} may be replaced by any set $\{a,b\}$ of two distinct elements, though the conclusion must then be adjusted accordingly. We work with $\{0,1\}$ for concreteness.
\end{remark}

\begin{proof}[Proof of Theorem~\ref{thm polynomial URS}]
Let $q := f^{-1}(\{0,1\}) = g^{-1}(\{0,1\})$ as a multiset (set with multiplicities). 
If $f(x)$ has degree $d$, then $|q| = 2d$ and $g(x)$ must also have degree $d$. Let $q_{f,0}$, $q_{f,1}$ be the multisets $q_{f,0} := f^{-1}(0)$ and $q_{f,1} := f^{-1}(1)$, and define $q_{g,0}$, $q_{g,1}$ similarly for $g(x)$. Then 
\[
    q_{f,0} \cap q_{f,1} = q_{g,0} \cap q_{g,1} = \emptyset\hspace{.3in}
    \text{ and }\hspace{.3in}
    q_{f,0} \cup q_{f,1} = q = q_{g,0} \cup q_{g,1}.
\]
Let $r_1, r_2, s_1, s_2$ be the unique disjoint multisets such that
\[
    q_{f,0} = r_1 \sqcup s_1 \hspace{.3in} q_{f,1} = r_2 \sqcup s_2 \hspace{.3in}
    q_{g,0} = r_1 \sqcup s_2 \hspace{.3in} q_{g,1} = r_2 \sqcup s_1.
\]
Since $q_{f,0}$ and $q_{f,1}$ are the multisets of roots of $f(x)$ and $f(x) - 1$ respectively, it follows that $e_k(q_{f,0}) = e_k(q_{f,1})$ for $1\leq k < d$, where $e_k$ is the $k$th elementary symmetric function in $d$ variables. The Newton-Girard identities imply that the $k$th power sum symmetric function
    \[
        p_k(x_1,\ldots,x_d) := x_1^k + x_2^k + \ldots + x_d^k
    \]
may be expressed as an integral polynomial in the first $k$ elementary symmetric functions. Hence $p_k(q_{f,0}) = p_k(q_{f,1})$ for $1\leq k < d$. The same argument applied to $g(x)$ implies that $p_k(q_{g,0}) = p_k(q_{g,1})$ for $1\leq k < d$. Note that if $q = r \sqcup s$ for multisets $q, r, s$, then
$
    p_k(q) = p_k(r) + p_k(s).
$
Thus,
\begin{align*}
    p_k(r_1) + p_k(s_1) = p_k(q_{f,0}) &= p_k(q_{f,1}) = p_k(r_2) + p_k(s_2)\\
    p_k(r_1) + p_k(s_2) = p_k(q_{g,0}) &= p_k(q_{g,1}) = p_k(r_2) + p_k(s_1),
\end{align*}
for $1\leq k < d$. The above identities imply that $p_k(r_1) = p_k(r_2)$ and $p_k(s_1) = p_k(s_2)$
for $1 \leq k < d$. Since $r_1, r_2, s_1, s_2$ are multisets with strictly fewer than $d$ elements it follows that $r_1 = r_2$ and $s_1 = s_2$. (Note that this is where we are using our assumption that $K$ has characteristic $0$.)

Thus either $r_i = \emptyset$ or $s_i = \emptyset$, which is equivalent to either $(q_{f,0},q_{f,1}) = (q_{g,0},q_{g,1})$ or $(q_{f,0},q_{f,1}) = (q_{g,1},q_{g,0})$. Polynomials with two common fibers are equal, hence $f(x) = g(x)$ or $f(x) = 1 - g(x)$.
\end{proof}

\begin{remark}
A \emph{unique range set} for a family $\F$ of meromorphic functions is a set $S$ such that if $f^{-1}(S) = g^{-1}(S)$ as sets with multiplicity for $f, g \in \F$, then $f = g$. Theorem \ref{thm polynomial URS} implies that $\{0,1\}$ is as close to being a unique range set for the family $\F = \CC[x]$ as a set with two elements can be. See Chen \cite[Question 1.2]{chen} and the discussion that follows for a survey of results related to unique range sets.
\end{remark}

\begin{proof}[Proof of Proposition \ref{prop two image obstruction}]
We prove the contrapostive. Suppose that $q \in \confhat_{\P,d}(K) \cap \confhat_{\Q,d}(K)$, and let $f(x), g(x) \in K[x]$ be degree-at-most-$d$ polynomials realizing $\P$ and $\Q$ on $q$. Since $\P$ and $\Q$ are two-image portraits, there exist linear polynomials $\ell_f(x), \ell_g(x) \in K[x]$ such that $\ell_f(f(q)) = \ell_g(g(q)) = \{0,1\}$. Since $q$ contains $2d$ points, it follows that
\[
    (\ell_f \circ f)^{-1}(\{0,1\}) = (\ell_g \circ g)^{-1}(\{0,1\}) = q.
\]
Hence $\ell_f(f(x)) = \ell_g(g(x))$ or $\ell_f(f(x)) = 1 - \ell_g(g(x))$ by Theorem \ref{thm polynomial URS}. In either case, $f$ is left associate to $g$. Thus $f$ and $g$ have the same fiber partition on $K$; in particular, $\Pi_\P = \Pi_\Q$.
\end{proof}

\begin{remark}
Proposition~\ref{prop two image obstruction} may fail in positive characteristic. Let $d = 2$, and consider the pair of portraits $\{\P,\Q\}$ illustrated in Figure~\ref{fig:2img_counterex}. Each of $\P$ and $\Q$ is a two-image portrait, but their fiber partitions $\Pi_\P = \{\{1,2\},\{3,4\}\}$ and $\Pi_\Q = \{\{1,3\},\{2,4\}\}$ are different, so it follows from Proposition~\ref{prop two image obstruction} that $\langle \P,\Q\rangle$ has no realizations of degree at most $2$ in characteristic zero. However, $\langle \P,\Q\rangle$ does have realizations in characteristic $2$: Indeed, let $q = (0, 1, \omega, \omega^2) \in \conf^4(\FF_4)$, where $\omega$ is a root of $x^2 + x + 1 \in \FF_2[x]$. Then $q$ realizes $\P$ via $f(x) = \omega^2x^2 + \omega^2x$ and $\Q$ via $g(x) = x^2 + \omega x + \omega$.
\end{remark}

\begin{figure}
    \begin{tikzpicture}[shorten >= -2pt, scale=0.75]
    \node (0) at (0, 2) {$1$};
    \node (1) at (0, 0) {$2$};
    \node (w) at (2, 2) {$3$};
    \node (w2) at (2, 0) {$4$};
    \draw[-{Latex[length=1.4mm,width=2mm]}, thick, red] (0) edge [out=90, in=180, looseness=7] (0);
    \draw[-{Latex[length=1.4mm,width=2mm]}, thick, red] (1) edge (0);
    \draw[-{Latex[length=1.4mm,width=2mm]}, thick, red] (w) edge (w2);
    \draw[-{Latex[length=1.4mm,width=2mm]}, thick, red] (w2) edge [out=270, in=0, looseness=7] (w2);
    \draw[-{Latex[length=1.4mm,width=2mm]}, thick, dashed] (0) edge (w);
    \draw[-{Latex[length=1.4mm,width=2mm]}, thick, dashed] (1) edge [out=180, in=270, looseness=7] (1);
    \draw[-{Latex[length=1.4mm,width=2mm]}, thick, dashed] (w) edge [out=0, in=90, looseness=7] (w);
    \draw[-{Latex[length=1.4mm,width=2mm]}, thick, dashed] (w2) edge (1);
\end{tikzpicture}
    \caption{A pair $\{\P,\Q\}$ of two-image portraits with degree-at-most-$2$ realizations in characteristic $2$ but not in characteristic $0$. The portraits $\P$ and $\Q$ are indicated by solid red and dashed black arrows, respectively.}
    \label{fig:2img_counterex}
\end{figure}
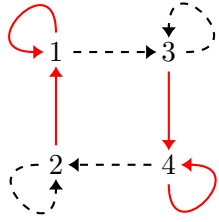

\begin{remark}
Unlike Theorem \ref{thm 2 image}, Proposition \ref{prop two image obstruction} cannot be directly generalized to rational realizations of three-image portraits: Consider the family of rational functions $f_t(x)$ with $t \in \CC$ defined by
\[
    f_t(x) = -\frac{x^2 - tx}{tx - 1}.
\]
If $t \neq \pm 1$, then $f_t$ is a degree-2 rational map with fixed points at $0$, $1$, and $\infty$. Each fixed point has an additional preimage: the non-fixed preimages of $0$, $1$, and $\infty$ are $t$, $-1$, and $1/t$, respectively. Thus, if $t \notin \{0, \pm 1\}$, then $f_t$ is an endomorphism of the configuration of six distinct points given by
\[
    q_t := (0, 1, \infty, t, -1, 1/t) \in \conf^6,
\]
with fiber partition $\Pi_{f_t} = \big\{\{0,t\},\{-1,1\},\{\infty, 1/t\}\big\}$.
If $t\ne \pm 1$, then $q_{1/t}$ and $q_t$ are different as configurations but equal as sets; hence $f_{1/t}$ is a degree-2 rational endomorphism of $q_t$ with a fiber partition different from that of $f_t$. Thus if $\P$ and $\Q$ are the three-image portraits corresponding to the action of $f_t$ and $f_{1/t}$ on $q_t$, then $\P$ and $\Q$ have distinct fiber partitions, yet
\[
    q_t \in \conf_{\P,2}(\wh{\CC}) \cap \conf_{\Q,2}(\wh{\CC}).
\]
\end{remark}

\section{Endomorphism semigroups of roots of unity}
\label{sec:endos}

Recall from the introduction that $\endo(q)$ is the semigroup of polynomial endomorphisms of a configuration $q \in \conf^n$; that is, polynomials $f(x)$ such that $f(q) \subseteq q$. If $d \geq 0$, then $\endo_d(q)$ is the degree-$d$ graded component of $\endo(q)$. In Question \ref{quest:intro max cardinality} we asked for the maximal cardinality $E_{n,d}$ of $\endo_d(q)$ as $q$ varies over $\conf^n$ and $0 \leq d < n - 1$. Theorem \ref{thm 2 image}(3) gives us the first nontrivial lower bound of
\[
    E_{2d,d} \geq 2d(2d - 1)
\]
coming from any configuration supporting a two-image portrait. Our survey of all quadratic pairs of portraits on 4 points gives the following determination of $E_{4,2}$.

\begin{prop}
\label{prop:4th_rou}
The maximal cardinality of $\endo_2(q)$ as $q$ varies over $\conf^4(\CC)$ is $E_{4,2} = 28$. This cardinality is achieved by a unique affine equivalence class of configurations, namely the class represented by $q = \mu_4 := (1,i,-1,-i)$, the 4th roots of unity.
\end{prop}

\begin{proof}
Let $q \in \conf^4(\CC)$ be a configuration with $m := |\endo_2(q)|$ quadratic endomorphisms. Then there are portraits $\P_1, \P_2, \ldots, \P_m$ on $[4]$ such that $q \in \bigcap_{i=1}^m \conf_{\P_i,2}.$ If $m$ is sufficiently large, then $\bigcap_{i=1}^m \conf_{\P_i,2}/\aff_1$ is 0-dimensional, and a computation shows that $m > 12$ suffices. (Note that by Theorem \ref{thm 2 image}, 12 quadratic two-image portraits with the same fiber partition have identical realization spaces.) In that case, the affine equivalence class of $q$ belongs to one of the finitely many portrait moduli spaces $\M_{\P,\Q,2}$ of dimension 0, hence we have finitely many candidates for configurations $q$ maximizing $|\endo_2(q)|$. For each of these candidates $q$, one may compute the degree of the unique degree-at-most-$2$ realization (if it exists) of each of the $4^4$ portraits on $q$. This procedure yields the configuration $q = \mu_4 = (1,i,-1,-i)$ as the unique affine equivalence class of complex configurations maximizing $|\endo_2(q)|$, with a total of 28 quadratic endomorphisms, which we construct explicitly below. We note that the configuration $q = (0, 1, 1 + \sqrt{2}, 2 + \sqrt{2})$ comes in second place with 20 quadratic endomorphisms.

One may check that the three polynomials
\[
    g_1(x) := x^2,
    \quad g_2(x) := \frac{i + 1}{2}\left(x^2 + i\right),
    \quad\text{and}\quad g_3(x) := \frac{1}{2}\left(x^2 + (i - 1)x + i\right)
\]
are endomorphisms of $\mu_4$.
Furthermore, since $ix$ is a linear automorphism of $\mu_4$, each of the polynomials
$i^j g_k(i^\ell x)$
is an endomorphism of $\mu_4$. There are $28$ such maps, hence these are all quadratic endomorphisms of $\mu_4$.
\end{proof}

\begin{remark}
Note that both $(1,i,-1,-i)$ and $(0, 1, 1 + \sqrt{2}, 2 + \sqrt{2})$ support two-image portraits, hence Theorem \ref{thm 2 image}(3) delivers $2d(2d - 1) = 12$ of their quadratic endomorphisms. Moreover, each of these configurations has nontrivial automorphisms: the configuration $(1,i,-1,-i)$ has the nontrivial automorphisms $i^kx$ for $1 \le k \le 3$, and $(0, 1, 1 + \sqrt 2, 2 + \sqrt 2)$ has the nontrivial automorphism $-x + 2 + \sqrt 2$.
\end{remark}

In retrospect, it is clear that configurations with an exceptional number of linear automorphisms should have more than the average number of degree-$d$ endomorphisms, since any one such endomorphism gives rise to more by pre- and post-composing with automorphisms as in Proposition \ref{prop:4th_rou}. 
The automorphism group of any configuration is finite cyclic---typically trivial---so
in the interest of finding configurations in $\conf^n(\CC)$ with many low degree endomorphisms, a natural starting point is to consider the family $\mu_n$ of configurations of $n$th roots of unity in $\CC$, since $\mu_n$ is the unique smallest configuration (up to affine conjugation) with an automorphism group of order $n$. Note that the endomorphism semigroup of $q$ is independent of the ordering of the points in the configuration, hence the ordering of $\mu_n$ is immaterial.

By Lagrange interpolation, each of the $n^n$ portraits on $n$ points is realized on $\mu_n$ by some polynomial of degree at most $n - 1$. In Table \ref{tab:rou endos} (which is the same as Table~\ref{tab:rou endos intro} in the introduction) we show the size of $\endo_d(\mu_n)$ for $3 \leq n \leq 8$ and $0 \leq d \leq n - 1$.

\begin{table}[h]
\centering
    \caption{The cardinality of $\endo_d(\mu_n)$ for $3 \le n \le d$ and $0 \le d \le n - 1$}
    \label{tab:rou endos}
    \begin{tabular}{c||ccccccccc}
    \diagbox[height=20pt]{$n$}{$d$} & $0$ & $1$ & $2$  & $3$   & $4$    & $5$     & $6$      & $7$\\\hline
    $3$ & $3$ & $3$ & $21$ &       &        &         &          &\\
    $4$ & $4$ & $4$ & $28$ & $220$ &        &         &          &\\
    $5$ & $5$ & $5$ & $5$  & $105$ & $3005$ &         &          &\\
    $6$ & $6$ & $6$ & $6$  & $30$  & $1992$ & $44616$ &          &\\
    $7$ & $7$ & $7$ & $7$  & $7$   & $105$  & $4907$  & $818503$ &\\
    $8$ & $8$ & $8$ & $8$  & $8$   & $280$  & $2968$  & $186840$ & $16587096$
    \end{tabular}
\end{table}

Let $\zeta_n \in \mu_n$ be a primitive $n$th root of unity, and note that for each degree $d$, the $n$ monomials $\zeta_n^k x^d$ belong to $\endo_d(\mu_n)$. Table \ref{tab:rou endos} suggests that these are the only degree-$d$ endomorphisms of $\mu_n$ for $d < n/2$. We prove this in Theorem \ref{thm rou endos} below.

\begin{thm}
\label{thm rou endos}
If $n > 2d \geq 1$, then the only degree-$d$ polynomial endomorphisms of $\mu_n$ in $\CC[x]$ are of the form $\zeta_n^k x^d$ for some $k\geq 0$.
\end{thm}

\begin{proof}
Let $\C$ denote the unit circle in $\CC$ and let $f(x) \in \CC[x]$ be a degree-$d$ polynomial. We claim that if there are more than $2d$ points $\alpha \in \C$ for which $f(\alpha) \in \C$, then $f(x) = \xi x^d$ for some $\xi \in \C$. The result then follows: if $f(\mu_n) \subseteq \mu_n$ with $n > 2d$, then $f(x) = \xi x^d$ and $\xi = f(1) \in f(\mu_n) \subseteq \mu_n$.

We now prove the claim. This is a special case of a result due to Cargo and Schneider \cite[Theorem 1]{cargo/schneider}; we provide a (slightly different) proof for completeness. Consider the M\"obius transformation
\[
    \ell(x) = i\Big(\frac{x - i}{x + i}\Big).
\]
Then $\ell(\C) = \wh{\RR} := \RR \cup \{\infty\}$, and the complex conjugate of $\ell(x)$ is
\begin{equation}\label{eq:mobius_conjugate}
    \overline{\ell}(x) = -i\Big(\frac{x + i}{x - i}\Big) = \ell^{-1}(x).
\end{equation}
Let $g(x) := \ell\circ f \circ \ell^{-1}$, and note that $\alpha, f(\alpha) \in \C$ if and only if $\ell(\alpha), g(\ell(\alpha)) \in \wh{\RR}$.

Fix $\beta \in \wh{\RR}$, and observe that $g(\beta) \in \wh{\RR}$ if and only if $g(\beta) = \overline{g}(\beta)$. Let $G(x,y) := [G_1(x,y) : G_2(x,y)]$ be the homogenization of $g$, where $G_1(x,y), G_2(x,y)$ are coprime, homogeneous polynomials of degree $d$, and write $\beta = [\beta_1 : \beta_2]$ as a point on $\PP^1(\CC)$ in homogeneous coordinates. Then $g(\beta) = \overline{g}(\beta)$ if and only if $[\beta_1 : \beta_2]$ is a zero of the degree-$2d$ form $G_1(x,y)\overline{G_2}(x,y) - \overline{G_1}(x,y)G_2(x,y)$. Thus, there are at most $2d$ distinct points $\beta \in \wh{\RR}$ with $g(\beta) \in \wh{\RR}$, hence at most $2d$ points $\alpha \in \C$ with $f(\alpha) \in \C$.

It follows from the previous paragraph that if there are more than $2d$ points $\alpha\in\C$ such that $f(\alpha) \in \C$, then $g(x) = \overline{g}(x)$. In this case, we have
\[
    \ell\circ f\circ \ell^{-1} = g = \overline{g} = \overline{\ell}\circ \overline{f}\circ \overline{\ell^{-1}} = \ell^{-1}\circ \overline{f} \circ \ell,
\]
where the last equality follows from \eqref{eq:mobius_conjugate}. The above equation implies that
$\ell^2\circ f = \overline{f}\circ \ell^2$,
and since $\ell^2(x) = 1/x$, we have
\[
    1 = f(x)\overline{f}(1/x).
\]
If $f(x) = c_d x^d + c_{d-1}x^{d-1} + \ldots + c_ex^e$, with $e \le d$ and $c_e \ne 0$, then the leading term of $f(x)\overline{f}(1/x)$ is $c_d\overline{c_e} x^{d - e}$. Thus $d = e$, $c_d\overline{c_d} = c_d\overline{c_e} = 1$, and $f(x) = c_dx^d = \xi x^d$ for some $\xi \in \C$, as claimed.
\end{proof}

Theorem~\ref{thm rou endos} is sharp in the sense that for each $d \ge 2$, there exist degree-$d$ endomorphisms of $\mu_{2d}$ that are not monomials. The polynomial $x^d$ realizes a two-image portrait on $\mu_{2d}$, since if $\zeta \in \mu_{2d}$, then $\zeta^d \in \{\pm 1\} \subseteq \mu_{2d}$. Thus, Theorem \ref{thm 2 image}(3) implies that $\mu_{2d}$ has at least $2d(2d - 1)$ degree-$d$ endomorphisms. These may be explicitly constructed as $\ell(x^d)$ where $\ell(x)$ is any linear polynomial such that $\ell(\pm 1) \in \mu_{2d}$, and only $2d$ of these endomorphisms are monomials. Note that Table \ref{tab:rou endos} implies that these two-image portrait realizations are the only cubic endomorphisms of $\mu_6$, but far from the only degree-4 endomorphisms of $\mu_8$.

\end{document}